%% file: multiderivative_relaxation__arXiv.tex
\documentclass[]{scrartcl}

\usepackage[utf8]{luainputenc}
\usepackage[USenglish]{babel}
\usepackage{csquotes}

\usepackage[a4paper,top=27mm,bottom=20mm,inner=25mm,outer=20mm]{geometry}

\usepackage[%
  backend=bibtex,bibencoding=ascii,
  style=numeric-comp,
  giveninits=true, uniquename=init, 
  natbib=true,
  url=true,
  doi=true,
  isbn=false,
  backref=false,
  maxnames=99,
  ]{biblatex}
\usepackage{amsmath}
\allowdisplaybreaks
\numberwithin{equation}{section}
\usepackage{amssymb}
\usepackage{commath}
\usepackage{mathtools}
\usepackage{bbm}
\usepackage{nicefrac}
\usepackage{subdepth}

\usepackage{siunitx}
\usepackage{algorithm}
\usepackage{algpseudocode}

\usepackage{amsthm}
\usepackage{thmtools}
\usepackage{etoolbox}
\makeatletter
\patchcmd{\thmt@setheadstyle}
 {\bgroup\thmt@space}
 {\thmt@space}
 {}{}
\patchcmd{\thmt@setheadstyle}
 {\egroup\fi}
 {\fi}
 {}{}
\makeatother
\declaretheoremstyle[
  bodyfont=\normalfont\itshape,
  headformat=\NAME\ \NUMBER\NOTE,
]{myplain}
\declaretheoremstyle[
  headformat=\NAME\ \NUMBER\NOTE,
]{mydefinition}
\newcommand{\envqed}{{\lower-0.3ex\hbox{$\triangleleft$}}}
\declaretheorem[style=myplain,numberwithin=section]{theorem}
\declaretheorem[style=myplain,numberlike=theorem]{lemma}

\declaretheorem[style=myplain,numberlike=theorem]{corollary}

\declaretheorem[style=mydefinition,numberlike=theorem,qed=\envqed]{remark}

\usepackage[plainpages=false,pdfpagelabels,hidelinks,unicode]{hyperref}

\usepackage{color}
\usepackage{graphicx}
\usepackage[small]{caption}
\usepackage{subcaption}

\begingroup\expandafter\expandafter\expandafter\endgroup
\expandafter\ifx\csname pdfsuppresswarningpagegroup\endcsname\relax
\else
\fi

\usepackage{pgfplotstable}

\usepgfplotslibrary{external}
\usepackage{placeins}

\usepackage{booktabs}
\usepackage{rotating}
\usepackage{multirow}

\usepackage{enumitem}

\usepackage{ifluatex}
\ifluatex
  \usepackage[no-math]{fontspec}
\else
  \usepackage[T1]{fontenc}
\fi
\usepackage{newpxtext,newpxmath}

\input{macros}

\newboolean{compilefromscratch}
\setboolean{compilefromscratch}{false}
\newcommand{\orcid}[1]{ORCID:~\href{https://orcid.org/#1}{#1}}
\usepackage{authblk}

\newenvironment{keywords}{\par\textbf{Key words.}}{\par}
\newenvironment{AMS}{\par\textbf{AMS subject classification.}}{\par}

\title{Multiderivative time integration methods preserving nonlinear functionals via relaxation}

\author[1]{Hendrik~Ranocha\thanks{\orcid{0000-0002-3456-2277}}}
\affil[1]{Institute of Mathematics, Johannes Gutenberg University Mainz, Staudingerweg 9, 55128 Mainz, Germany}

\author[2]{Jochen Schütz\thanks{\orcid{0000-0002-6355-9130}}}
\affil[2]{Faculty of Sciences \& Data Science Institute, Hasselt University, Agoralaan Gebouw D, BE-3590 Diepenbeek, Belgium}

\date{February 7, 2024} 

\makeatletter
\hypersetup{pdfauthor={Hendrik Ranocha, Jochen Schütz}} 
\hypersetup{pdftitle={Multiderivative time integration methods preserving nonlinear functionals via relaxation}} 
\makeatother

\begin{document}

\maketitle

\begin{abstract}
\noindent
  \input{abstract.tex}
\end{abstract}



\section{Introduction}
\label{sec:introduction}
\input{sec_introduction.tex}

\section{Basic ideas of multiderivative and relaxation methods}
\label{sec:basics}
  First, we present the basic ideas of multiderivative time integration methods
  as well as relaxation schemes. Thereafter, we will describe our approach to
  combine them in Section~\ref{sec:multiderivative-relaxation}.

  \subsection{Multi-derivative Runge-Kutta methods}
  \input{sec_multiderivative}

  \subsection{Relaxation methods}
  \input{sec_relaxation}

\section{Efficient application of relaxation to multiderivative methods}
\label{sec:multiderivative-relaxation}
  \input{sec_entropy_estimation.tex}

\section{Stability properties}
\label{sec:stability}
\input{sec_stability.tex}

\section{Numerical experiments}
\label{sec:numerical_experiments}
\input{sec_numres.tex}

\section{Summary and conclusions}
\label{sec:summary}
\input{sec_conout.tex}

\section*{Acknowledgments}

\input{funding}

\printbibliography

\appendix

\section{Butcher tableaux}

Several multiderivative Runge-Kutta schemes have been used in this work. For convenience, we list them here in this appendix.

\subsection{Third-, fourth-, and fifth-order schemes from Chan and Tsai \cite{chan2010explicit}}\label{sec:ctbutcher}

\begin{equation}
\label{eq:CT(3,2)}
  \text{CT(3,2): } \qquad
  \renewcommand*{\arraystretch}{1.4}
  \begin{array}{c | cc || cc}
    0 &  &  &  & \\
    1 & 1 &  & \frac{1}{2} & \\
    \hline
      &  \frac 2 3 & \frac 1 3  & \frac{1}{6} & 0
  \end{array}
\end{equation}

\begin{equation}
\label{eq:CT(4,2)}
  \text{CT(4,2): } \qquad
  \renewcommand*{\arraystretch}{1.4}
  \begin{array}{c | cc || cc}
    0 &  &  &  & \\
    \frac{1}{2} & \frac{1}{2} &  & \frac{1}{8} & \\
    \hline
      &  1 & 0  & \frac{1}{6} & \frac{1}{3}
  \end{array}
\end{equation}

\begin{equation}
\label{eq:CT(5,3)}
  \text{CT(5,3): } \qquad
  \renewcommand*{\arraystretch}{1.4}
  \begin{array}{c | ccc || ccc}
    0 &  &  &  & \\
    \frac{2}{5} & \frac{2}{5} & & & \frac{2}{25} & \\
    1 & 1 & & & \frac{-1}{4} & \frac 3 4 \\
    \hline
      &  1 & 0 &0  & \frac{1}{8} & \frac{25}{72} & \frac 1 {36}
  \end{array}
\end{equation}

\subsection{Three-derivative schemes from Turac{i} and \"{O}zi\c{s} \cite{TurTur2017}}
\label{sec:turturbutcher}
\begin{equation}
\label{eq:TT(5,2)}
  \text{TO(5,2): } \qquad
  \renewcommand*{\arraystretch}{1.4}
  \begin{array}{c | cc || cc || cc}
    0 &  &  &  & \\
    \frac 2 5 & \frac 2 5 & 0 & \frac 2 {25} & & \frac 4 {375} & \\
    \hline
      & 1 & 0 & \frac 1 2 & 0 & \frac 1 {16} & \frac 5 {48}
  \end{array}
\end{equation}

\begin{equation}
\label{eq:TT(7,3)}
  \text{TO(7,3): } \qquad
  \renewcommand*{\arraystretch}{1.4}
  \begin{array}{c | ccc || ccc || ccc}
    0 & & & & & & &  & & \\
    c_2 & c_2 & 0 & 0 & \frac{c_2^2}{2} & 0 & 0 & \frac{c_2^3}{6} \\
    c_3 & c_3 & 0 & 0 & \frac{c_3^2}{2} & 0 & 0 & \frac{c_3^3}{6}-a_{32} & a_{32} & \\
    \hline
      & 1 & 0 & 0 & \frac 1 2 & 0 & 0 & \frac 1 {30} & \frac 1 {15} + \frac{13\sqrt 2}{480} & \frac 1 {15} - \frac{13\sqrt{2}}{480}
  \end{array}
\end{equation}
Here, $c_2 = \frac{3-\sqrt{2}} 7$, $c_3 = \frac{3+\sqrt 2}{7}$ and $a_{32} = \frac{122 + 71\sqrt{2}}{7203}$.

\subsection{SSP schemes from Gottlieb and co-workers \cite{gottlieb2022high}}
\label{sec:sspbutcher}
\begin{equation}
\label{eq:SSP3}
  \text{SSP-I2DRK3-2s: } \qquad
  \renewcommand*{\arraystretch}{1.4}
  \begin{array}{c | cc || cc}
    0 & 0 & 0 & -\frac 1 6 & 0\\
    1 & 0 & 1 & -\frac 1 6 & -\frac 1 3 \\
    \hline
      & 0 & 1 & -\frac 1 6 & -\frac 1 3
  \end{array}
\end{equation}

The coefficients for SSP-I2DRK4-5s are not as nice as for the above examples; we refer the reader to \cite[Section~2.3, Paragraph ``Fourth order'']{gottlieb2022high}.

\subsection{Some collocation-based Hermite-Birkhoff Runge-Kutta schemes}
\label{sec:hbbutcher}

\begin{table}[h!]
\centering
\caption{Runge-Kutta table of the continuous HB-I2DRK6-3s method. The method is uniformly sixth-order in $0 \leq \theta \leq 1$. As usual, the original Runge-Kutta method is obtained by setting $\theta = 1$.  Note that the Butcher tableau corresponding to the second derivative has been printed under the one for the first derivative.}\label{tbl:contHB-I2DRK6-3s}
\begin{tabular}{c|ccc}
$0$ & $0$   & $0$   & $0$    \\
$1/2$ & $101/480$   & $8/30$ & $55/2400$ \\
$1$ & $7/30$   & $16/30$   & $7/30$ \\ \hline &
$4\,\theta ^6-\frac{68\,\theta ^5}{5}+\frac{33\,\theta ^4}{2}-\frac{23\,\theta ^3}{3}+\theta$ &
$\frac{8\,\theta ^3\,\left(6\,\theta ^2-15\,\theta +10\right)}{15}$ &
$-\frac{\theta ^3\,\left(120\,\theta ^3-312\,\theta ^2+255\,\theta -70\right)}{30}$
\\ \hline
  &$0$  & $0$   & $0$ \\
  &$65/4800$ & $-25/600$ & $-25/8000$ \\
  & $5/300$ & $0$ & $-5/300$ \\ \hline &
$\frac{\theta ^2\,\left(40\,\theta ^4-144\,\theta ^3+195\,\theta ^2-120\,\theta +30\right)}{60}$ &
$\frac{8\,\theta ^3\,{\left(\theta -1\right)}^3}{3}$ &
$\frac{\theta ^3\,\left(40\,\theta ^3-96\,\theta ^2+75\,\theta -20\right)}{60}$
\end{tabular}
\end{table}

One set of coefficients is given in \autoref{tbl:contHB-I2DRK6-3s}.
The coefficients of the other collocation-based Hermite-Birkhoff Runge-Kutta schemes
used in this work can be generated by the MATLAB code contained in our
reproducibility repository \cite{ranocha2023multiderivativeRepro}.

\end{document}

%% file: macros.tex
\usepackage{gensymb} 

\usepackage{xparse}

\let\epsilon\varepsilon
\let\phi\varphi
\let\rho\varrho

\DeclareMathOperator{\mmod}{mod}

\usepackage{xspace}

\renewcommand{\O}{\mathcal{O}}

\providecommand\R{}
\renewcommand{\R}{\mathbb{R}}
\providecommand\C{}
\renewcommand{\C}{\mathbb{C}}

\newcommand{\etanew}{\eta^{\mathrm{new}}}

\NewDocumentCommand{\RK}{o m O{\the\numexpr#2-1\relax} m O{} O{} o}{%
  \IfValueTF{#1}{#1}{RK}%
  #2(#3)#4%
  \ifblank{#6}{}{\textsubscript{F}}%
  \ifblank{#5}{}{[#5]}%
  \IfValueT{#7}{#7}%
}

\newcommand{\dt}{\Delta t}

\renewcommand{\vec}[1]{\pmb{#1}}

\NewDocumentCommand{\opD}{m+g}{%
  \IfNoValueTF{#2}
    {D_{#1}}
    {D_{#1,#2}}%
}
\NewDocumentCommand{\opDsplit}{m+g}{%
  \IfNoValueTF{#2}
    {\widetilde{D}_{#1}}
    {\widetilde{D}_{#1,#2}}%
}
\NewDocumentCommand{\opM}{g}{%
  \IfNoValueTF{#1}
    {M}
    {M_{#1}}%
}
\NewDocumentCommand{\opQ}{g}{%
  \IfNoValueTF{#1}
    {Q}
    {Q_{#1}}%
}
\NewDocumentCommand{\opI}{g}{%
  \IfNoValueTF{#1}
    {I}
    {I_{#1}}%
}
\NewDocumentCommand{\opV}{g}{%
  \IfNoValueTF{#1}
    {V}
    {V_{#1}}%
}
\NewDocumentCommand{\opB}{g}{%
  \IfNoValueTF{#1}
    {B}
    {B_{#1}}%
}
\NewDocumentCommand{\opR}{g}{%
  \IfNoValueTF{#1}
    {R}
    {R_{#1}}%
}
\NewDocumentCommand{\opN}{m+g}{%
  \IfNoValueTF{#2}
    {N_{#1}}
    {N_{#1,#2}}%
}

\NewDocumentCommand{\fnum}{g}{%
  \IfNoValueTF{#1}
    {f^{\mathrm{num}}}
    {f^{\mathrm{num,#1}}}%
}
\NewDocumentCommand{\vecfnum}{g}{%
  \IfNoValueTF{#1}
    {\vec{f}^{\mathrm{num}}}
    {\vec{f}^{\mathrm{num,#1}}}%
}
\NewDocumentCommand{\vecfcorr}{g}{%
  \IfNoValueTF{#1}
    {\vec{f}^{\mathrm{corr}}}
    {\vec{f}^{\mathrm{corr,#1}}}%
}
\NewDocumentCommand{\fvol}{g}{%
  \IfNoValueTF{#1}
    {f^{\smash{\mathrm{vol}}}}
    {f^{\smash{\mathrm{vol,#1}}}}%
}

\definecolor{col1}{RGB}{         0  135.4688  255.0000}
\definecolor{col2}{RGB}{         0    9.9609  255.0000}
\definecolor{col3}{RGB}{  115.5469         0  255.0000}
\definecolor{col4}{RGB}{  241.0547         0  255.0000}
\definecolor{col5}{RGB}{  255.0000         0  143.4375}
\definecolor{col6}{RGB}{  255.0000         0   17.9297}

\pgfplotscreateplotcyclelist{hierarchy6}{%
	{col1,thick,mark=o},
	{col2,thick,mark=pentagon},
	{col3,thick,mark=square},
	{col4,thick,mark=triangle},
	{col5,thick,mark=x},
   {col6,thick,mark=otimes}}

%% file: abstract.tex
We combine the recent relaxation approach with multiderivative Runge-Kutta methods
to preserve conservation or dissipation of entropy functionals for
ordinary and partial differential equations. Relaxation methods are
minor modifications of explicit and implicit schemes, requiring only the solution
of a single scalar equation per time step in addition to the baseline scheme.
We demonstrate the robustness of the
resulting methods for a range of test problems including the 3D compressible Euler
equations. In particular, we point out improved error growth rates for certain
entropy-conservative problems including nonlinear dispersive wave equations.

%% file: sec_introduction.tex
Preserving the correct evolution of a (nonlinear) functional of the solution
of a differential equation is an important task in many areas. For ordinary
differential equations (ODEs), such a structure-preserving approach is
important, e.g., as energy-conserving method for Hamiltonian problems.
In general, it falls into the realm of geometric numerical integration
\cite{hairer2006geometric}.

In this article, we concentrate on the preservation of functionals under multiderivative
time discretization, including the conservation of invariants as well as
dissipative problems. Building on our previous work \cite{ranocha2023functional},
we use the relaxation approach developed recently in
\cite{ketcheson2019relaxation,ranocha2020relaxation,ranocha2020general}.
The first ideas of relaxation go back to \cite{sanzserna1982explicit}
and have been refined
since then. Originally developed for Runge-Kutta and linear multistep methods,
the general theory of \cite{ranocha2020general} can be applied to a wide range
of time integration schemes including also deferred correction and ADER methods
\cite{abgrall2022relaxation,gaburro2023high} as well as IMEX schemes
\cite{kang2022entropy,li2022implicit,li2023linearly}. The construction can be
extended to multiple invariants \cite{biswas2023multiple,biswas2023accurate}
and implicit schemes not relying on an exact solution of the equations
\cite{linders2023resolving}.
Relaxation methods have been applied successfully to Hamiltonian problems
\cite{ranocha2020relaxationHamiltonian,zhang2020highly},
kinetic equations \cite{leibner2021new},
entropy-stable methods for compressible fluid flows
\cite{waruszewski2022entropy,ranocha2020fully},
and nonlinear dispersive wave equations
\cite{ranocha2021broad,mitsotakis2021conservative,ranocha2021rate}.

Multiderivative schemes are numerical integration schemes for differential equations that can be traced back to at least \cite{Turan1950} in the context of quadrature rules; for ODEs, they have been discussed in \cite{KastlungerWanner1972,HaWa73}.
We do not attempt to give a full historic overview, but refer the reader to \cite[Section 2.1]{seal2014high} for one such thorough overview.
In contrast to more established schemes, multiderivative schemes do not only use the ODE's right-hand side, but also temporal derivatives thereof. Recently, these methods have gained renewed interest, as the construction principle allows for very flexible schemes, be it with respect to SSP properties, e.g., \cite{gottlieb2022high,AAF2020}, efficiency, e.g., \cite{seal2014high,abdi2020implementation,schutz2022parallel,ZEIFANG2023128198}, stability, e.g.,  \cite{chan2010explicit,schutz2021asymptotic,zeifang2022stability} and many others. In this work, we are to our knowledge the first to combine relaxation with the multiderivative approach; results with up to four derivatives will be shown.

In the following Section~\ref{sec:basics}, we briefly review the construction
of multiderivative Runge-Kutta methods and the relaxation approach.
Section~\ref{sec:multiderivative-relaxation} is devoted to the development
of appropriate entropy estimates for multiderivative methods applied to
non-conservative problems. In Section~\ref{sec:stability}, we discuss stability
properties of the relaxed methods. Next, we present numerical experiments in
Section~\ref{sec:numerical_experiments}, ranging from convergence experiments
to robustness demonstrations for PDE discretizations as well as qualitative and
quantitative improvements via reduced error growth rates. Finally, we summarize
the results in Section~\ref{sec:summary}.

%% file: sec_multiderivative.tex
To introduce the methods considered in this paper, we define the initial-value problem
\begin{equation}
\label{eq:ODE}
  u'(t) = f(u(t)),
  \quad
  u(0) = u^0
\end{equation}
for given right-hand side $f\colon \R^d \rightarrow \R^d$ and initial condition $u^0 \in \R^d$.
The basic idea of multiderivative methods is to use not only $u'(t)$, which equals $f(u(t))$,
but also higher derivatives such as $u''(t)$, $u'''(t)$, \dots to compute
a numerical approximation.
All these quantities can be computed using the available data, e.g., $u''(t) = f'(u)f(u)$. (Please note that for the ease of presentation, we from now on refrain from explicitly writing the $t-$dependency of $u$ in, e.g., terms as $f(u)$.)
The most well-known class of schemes using this
information are Taylor series methods, which just approximate the solution
in one time step by a truncated Taylor series. In this article, we are
instead interested in multiderivative methods using additionally the idea to
compute multiple internal stages as in Runge-Kutta methods.

Define the higher-order temporal derivatives of $f(u)$ through
\begin{align*}
g^{(1)}(u) := f(u), \quad g^{(2)}(u) := f'(u) f(u),
\quad g^{(3)}(u) := f'(u) f'(u) f(u) + f''(u)\bigl( f(u), f(u) \bigr)
\end{align*}
and so on.
Note that for the exact solution to \eqref{eq:ODE}, there holds $u^{(k)}(t) = g^{(k)}(u)$, where $u^{(k)}(t)$ denotes the $k-$th temporal derivative of $u$.
Please note that it gets increasingly complex to compute the quantities $g^{(k)}(u)$; we hence limit our numerical computations to four derivatives. With this
notation, a multi-derivative Runge-Kutta method with $s$ stages $y^i$ can be
written as, e.g., in \cite{KastlungerWanner1972},
\begin{equation}
\label{eq:multi-derivative-RK}
\begin{aligned}
  y^i &= u^n + \sum_{k=1}^m \dt^k \sum_{j=1}^s a^{(k)}_{ij} g^{(k)}(y^j),
  \\
  u^{n+1} &= u^n + \sum_{k=1}^m \dt^k \sum_{i=1}^s b^{(k)}_i g^{(k)}(y^i).
\end{aligned}
\end{equation}
$m$ is the number of derivatives taken into account (there holds $m=1$ for the classical Runge-Kutta schemes); the coefficients $a^{(k)}_{ij}$ and $b^{(k)}_i$ can as usual be represented through an extended Butcher tableau
\begin{equation}
  \renewcommand*{\arraystretch}{1.4}
  \begin{array}{c | c || c || c || c}
    c & A^{(1)} & A^{(2)} & \ldots & A^{(m)} \\
    \hline
      & b^{(1)} & b^{(2)} & \ldots & b^{(m)}
  \end{array}
\end{equation}
where $A^{(k)} = (a^{(k)}_{ij})_{ij} \in \R^{s \times s}$, $b^{(k)} = (b_i^{(k)})_{i} \in \R^s$, and
$c = (c_i)_{i} \in \R^s$ with $c_i = \sum_{j=1}^s a^{(1)}_{ij}$.

In this work, we will use the following classes of schemes:
\begin{itemize}
 \item The third-, fourth-, and fifth-order explicit schemes of Chan and Tsai \cite{chan2010explicit}. They belong to the class of two-derivative schemes; only their order $p$ and number of stages $s$ differ. They are referred to as CT($p$, $s$). For convenience, the Butcher tableaux have been reported in Sec.~\ref{sec:ctbutcher}. Note that CT(4, 2) has the same stability region as the original fourth-order Runge-Kutta method, but requires only one $g^{(1)}$ and two $g^{(2)}$ evaluations instead of four $g^{(1)}$ evaluations.
 \item The explicit three-derivative schemes of order $p = 5$ and $p=7$, respectively, of Turac{i} and \"{O}zi\c{s} \cite{TurTur2017}. They are referred to as TO($p$, $s$), with $s$ denoting again the number of stages. Butcher tableaux are reported in Sec.~\ref{sec:turturbutcher}.
 \item The third- and fourth-order two-derivative SSP schemes from Gottlieb and co-workers \cite{gottlieb2022high}; they are referred to as SSP-I2DRK3-2s and SSP-I2DRK4-5s, respectively. Also here, the Butcher tableaux can be found in Sec.~\ref{sec:sspbutcher}.
 \item Some examples of a class of collocation-based, fully implicit schemes, called Hermite-Birkhoff methods. We refer to these methods as HB-I$m$DRK$p$-$s$s, where again $m$ denotes the number of derivatives, $p$ denotes the order and $s$ the number of stages. Butcher tableaux of HB-I2DRK4-2s and HB-I2DRK6-3s can be found in Tables \ref{tbl:contHB-I2DRK4-2s} and \ref{tbl:contHB-I2DRK6-3s}, respectively (set $\theta = 1$); all the other methods can be generated through a Matlab code in our reproducibility repository \cite{ranocha2023multiderivativeRepro}.
 These methods are defined by prescribing the coefficients $c$ through an equidistant spacing of $[0, 1]$, i.e., $c_i := \frac {i-1} {s-1}$, $1 \leq i \leq s$. Then, a standard collocation approach is performed, with the exception that $m$ derivatives of $u$ are taken into account rather than only one. For more information, consult \cite{hairer2008solving} (for $m=1$) and \eqref{eq:collocation}.
 An exception to this construction principle is the HB-I2DRK3-2s method that uses Hermite-Birkhoff interpolation on $c = (0, 1)^T$ with one derivative at the left, and two derivatives at the right point. In two-point notation, it can be written as
   \begin{align*}
      u^{n+1} = u^n + \frac{\dt}{3} \left(g^{(1)}(u^n) + 2 g^{(1)}(u^{n+1})\right) - \frac{\dt^2}{6} g^{(2)}(u^{n+1});
   \end{align*}
 it is both $A-$ and $L-$stable.
\end{itemize}

%% file: sec_relaxation.tex
Next, we describe relaxation methods following the general theory of
\cite{ranocha2020general}, specialized to one-step methods.
Consider a one-step method computing a numerical solution $u^{n+1}$
at time $t^{n+1} = t^{n} + \dt$ from the previous step $u^{n}$.
Assume that there is a functional $\eta$ of the solution $u$ of
\eqref{eq:ODE} that should be preserved. In many cases, the functional $\eta$
can be interpreted as an ``energy'' or ``entropy'' of the solution.
In this paper, we will usually call $\eta$ an entropy.
Some important cases are invariants satisfying
$\frac{\dif}{\dif t} \eta\bigl( u(t) \bigr) = 0$
and dissipative problems where
$\frac{\dif}{\dif t} \eta\bigl( u(t) \bigr) \le 0$.

The basic idea of relaxation methods is to post-process a baseline
solution and continue the numerical integration using
\begin{equation}
  u^{n+1}_\gamma = u^{n} + \gamma \left( u^{n+1} - u^{n} \right)
\end{equation}
as the numerical approximation at time
\begin{equation}
  t^{n+1}_\gamma = t^{n} + \gamma \left( t^{n+1} - t^{n} \right)
\end{equation}
instead of the baseline solution $u^{n+1}$ at time $t^{n+1}$.
The relaxation parameter $\gamma$ is computed by solving the scalar
root finding problem
\begin{equation}
  \eta(u^{n+1}_\gamma)
  = \eta(u^{n}) + \gamma \bigl( \etanew - \eta(u^{n}) \bigr),
\end{equation}
where $\etanew$ is an estimate of the entropy at $t^{n+1}$, i.e.,
$\etanew \approx \eta(u(t^{n+1}))$.
For entropy-conservative problems, we use $\etanew = \eta(u^{n})$,
ensuring discrete entropy conservation of the form
$\eta(u^{n+1}_\gamma) = \eta(u^{n})$.
For dissipative problems, we require $\etanew \le \eta(u^{n})$,
resulting in discrete entropy dissipation of the form
$\eta(u^{n+1}_\gamma) \le \eta(u^{n})$.
We will describe how to obtain such estimates for multiderivative
methods in the following Section~\ref{sec:multiderivative-relaxation}.

\begin{remark}
 \input{sec_relaxation_parameter.tex}
\end{remark}

The theory of \cite{ranocha2020general} yields the following result:
\begin{theorem}
  Consider the relaxation procedure described above for a one-step
  method of order $p \ge 2$ with exact value at time $t^{n}$, i.e.,
  $u^{n+1} = u(t^{n+1}) + \O(\dt^{p+1})$.
  If the time step $\dt$ is sufficiently small,
  $\etanew = \eta\bigl( u(t^{n+1}) \bigr) + \O(\dt^{p+1})$, and
  \begin{equation}
  \label{eq:entropy-non-degenerate}
    \eta'(u^{n+1}) \frac{u^{n+1} - u^{n}}{\| u^{n+1} - u^{n} \|}
    =
    c \dt + \O( \dt^{2} ),
    \quad
    \text{with } c \neq 0,
  \end{equation}
  there is a unique solution $\gamma = 1 + \O(\dt^{p-1})$ and the
  relaxation method satisfies
  $u^{n+1}_\gamma = u(t^{n+1}_\gamma) + \O(\dt^{p+1})$.
\end{theorem}
Thus, the relaxation approach keeps at least the order of accuracy of the
baseline method (and can sometimes even lead to superconvergence, e.g.,
\cite{ranocha2020relaxationHamiltonian}). The non-degeneracy condition
\eqref{eq:entropy-non-degenerate} basically requires that the entropy
$\eta$ is ``nonlinear enough'', e.g., strictly convex --- since the relaxation
approach conserves all linear invariants of the baseline method, in contrast
to the orthogonal projection method \cite[Section~IV.4]{hairer2006geometric}.

The relaxation approach can be interpreted as a line search along the
secant connecting $u^{n}$ and $u^{n+1}$. To keep at least the order of
accuracy of the baseline method, relaxation methods require adapting the
time as well. Otherwise, they result in the incremental direction technique
(IDT) of \cite{calvo2006preservation}, losing an order of accuracy for
Runge-Kutta methods \cite{calvo2006preservation,ketcheson2019relaxation}
and sometimes even more for multistep methods \cite{ranocha2020general}.

%

%% file: sec_relaxation_parameter.tex
If the entropy is given by an inner product norm, i.e.,
$\eta(u) = \| u \|^2 = \langle u, u \rangle$, the relaxation parameter
can be computed explicitly \cite{ketcheson2019relaxation,ranocha2020general}.
For a conservative problem,
\begin{equation}
  \gamma = -2 \frac{\langle u^{n}, u^{n+1} - u^{n} \rangle}
                   {\| u^{n+1} - u^{n} \|^2},
\end{equation}
while a general (e.g., dissipative) problem yields
\begin{equation}
  \gamma = \frac{\etanew - \eta(u^{n}) - 2 \langle u^{n}, u^{n+1} - u^{n} \rangle}
                {\| u^{n+1} - u^{n} \|^2}.
\end{equation}

If general entropies are considered, one has to rely on \emph{scalar} root-finding algorithms such as Newton-Raphson or bisection.

%% file: sec_entropy_estimation.tex
Multiderivative methods can be combined directly with relaxation methods
to conserve an invariant. For more general problems where the entropy is
not necessarily conserved, e.g., dissipative problems, we need to develop
an appropriate estimate $\etanew$ satisfying
\begin{equation}
\label{eq:etanew-integral}
  \etanew
  =
  \eta\bigl( u(t^{n+1}) \bigr)
  + \O(\dt^{p+1})
  =
  \eta\bigl( u(t^{n}) \bigr)
  + \int_{t^{n}}^{t^{n+1}} \underbrace{(\eta' f)\bigl( u(t) \bigr)}_{= \frac{\dif}{\dif t} \eta(u(t))} \dif t
  + \O(\dt^{p+1}).
\end{equation}
For Runge-Kutta methods, an efficient approach is to use the Runge-Kutta
quadrature directly to approximate the integral in \eqref{eq:etanew-integral}
\cite{ketcheson2019relaxation,ranocha2020relaxation}, i.e.,
\begin{equation}
  \etanew = \eta(u^{n}) + \dt \sum_{i=1}^s b_i (\eta' f)(y^i).
\end{equation}
This yields an appropriately dissipative estimate $\etanew \le \eta(u^{n})$
for dissipative problems if the Runge-Kutta weights $b_i$ are non-negative.

However, we cannot use such an approach with multiderivative methods since
we only have information on the first derivative of $\eta \circ u$ in general;
even if we know that $\frac{\dif}{\dif t} \eta\bigl( u(t) \bigr) \le 0$
for a dissipative problem, we typically have no information on the second derivative
$\frac{\dif^2}{\dif t^2} \eta\bigl( u(t) \bigr)$. Thus, we cannot
use the step update formula of a multiderivative method to compute an
appropriate estimate $\etanew$ directly.

\subsection{Entropy estimates (not only) for dissipative problems}
\label{sec:entropy-estimates}

Here, we follow the approach of \cite[Section~3.3]{ranocha2020general},
see also \cite{calvo2010projection}. The idea is to approximate the
integral in \eqref{eq:etanew-integral} by a quadrature rule of order
at least $p$ with nodes $\tau_i \in [t^{n}, t^{n+1}]$ and positive
weights $w_i > 0$, resulting in
\begin{equation}
\label{eq:etanew-quadrature}
  \etanew = \eta(u^{n}) + \dt \sum_{i} w_i (\eta' f)\bigl( y(\tau_i) \bigr),
\end{equation}
where $y(\tau_i)$ are the values of an interpolant $y$ of order at least
$p - 1$ at the quadrature nodes (i.e., a dense output).

\begin{lemma}
Assuming that the quadrature weights $w_i$ are positive and that $\eta'(u) f(u) \leq 0$, the expression \eqref{eq:etanew-quadrature} yields a dissipative estimate for dissipative functionals, i.e.,
\begin{align*}
 \etanew \leq \eta(u^n).
\end{align*}
\end{lemma}

%
In this work, we use Gauss-Lobatto-Legendre
quadrature to compute the quadrature weights $w_i$ and points $\tau_i$ in \eqref{eq:etanew-quadrature}\footnote{%
We could also use a Gauss-Legendre quadrature with one node less to get
the same order of accuracy of the quadrature. However, since we often have
the values of $f$ at the endpoints, we save an evaluation of the right-hand
side $f$ by using Gauss-Lobatto-Legendre nodes.}. It remains to determine the values $y(\tau_i$). Two approaches have been investigated here:

\paragraph{Hermite-Birkhoff interpolation}
Many multiderivative methods compute already $f(u^{n}) = g^{(1)}(u^{n})$,
$g^{(2)}(u^{n})$, $f(u^{n+1}) = g^{(1)}(u^{n+1})$, and $g^{(2)}(u^{n+1})$.
Thus, we can use them together
with $u^{n}$ and $u^{n+1}$ to compute a quintic Hermite interpolation of
the numerical solution.  For the special case of
four quadrature nodes (which is exact for polynomials of degree $\le 5$),
this results in the nodes
\begin{equation}
  \tau_1 = t^n, \;
  \tau_2 = t^n + \left( \frac{1}{2} - \frac{\sqrt{5}}{10} \right) \dt, \;
  \tau_3 = t^n + \left( \frac{1}{2} + \frac{\sqrt{5}}{10} \right) \dt, \;
  \tau_4 = t^n + \dt = t^{n+1},
\end{equation}
the weights
\begin{equation}
  w_1 = \frac{1}{12}, \;
  w_2 = \frac{5}{12}, \;
  w_3 = \frac{5}{12}, \;
  w_4 = \frac{1}{12},
\end{equation}
and the interpolants
\begin{equation}
\label{eq:entropyestimation}
\begin{aligned}
  y(\tau_1)
  &=
  u^{n},
  \\
  y(\tau_2)
  &=
    \frac{250 + 82 \sqrt{5}}{500} u^{n}
  + \frac{ 60 + 16 \sqrt{5}}{500} \dt f(u^{n})
  + \frac{  5 +    \sqrt{5}}{500} \dt^2 g^{(2)}(u^{n})
  \\
  &\quad
  + \frac{250 - 82 \sqrt{5}}{500} u^{n+1}
  + \frac{-60 + 16 \sqrt{5}}{500} \dt f(u^{n+1})
  + \frac{  5 -    \sqrt{5}}{500} \dt^2 g^{(2)}(u^{n+1}),
  \\
  y(\tau_3)
  &=
    \frac{250 - 82 \sqrt{5}}{500} u^{n}
  + \frac{ 60 - 16 \sqrt{5}}{500} \dt f(u^{n})
  + \frac{  5 -    \sqrt{5}}{500} \dt^2 g^{(2)}(u^{n})
  \\
  &\quad
  + \frac{250 + 82 \sqrt{5}}{500} u^{n+1}
  + \frac{-60 - 16 \sqrt{5}}{500} \dt f(u^{n+1})
  + \frac{  5 +    \sqrt{5}}{500} \dt^2 g^{(2)}(u^{n+1}),
  \\
  y(\tau_4)
  &=
  u^{n+1}.
\end{aligned}
\end{equation}

\paragraph{Continuous Runge-Kutta output}
The class of Hermite-Birkhoff multiderivative Runge-Kutta methods is
collocation-based. In a nutshell --- for more information on collocation schemes, see \cite{hairer2008solving} --- a collocation scheme is defined as a polynomial $y_f$ of order $p-1$ that interpolates the derivative $f\equiv g^{(1)}$ at collocation points $t^n + \dt c_1, \ldots, t^n + \dt c_s$. In the case of the Hermite-Birkhoff methods, interpolation has to be understood in the sense of Hermite interpolation. Then, the RK stages and update are defined as
\begin{alignat}{2}
 \label{eq:collocation}
 y^{i}   &= u^n + \int_{t^n}^{t^n + \dt c_s}  y_f(\tau) \mathrm d \tau, \qquad
 u^{n+1} &\ =\ & u^n + \int_{t^n}^{t^{n+1}} y_f(\tau) \mathrm{d} \tau.
\end{alignat}
Using suitable Lagrange basis functions for Hermite interpolation,
see \cite[Sec.~8.5]{QuarteroniNumA}, these expressions can be
written in the form of a linear combination of the $g^{(k)}$, $1 \leq k \leq m$.
The same idea can be used to produce continuous output, i.e., produce values $y_{\theta} \approx u(t^n + \theta \dt)$ for $\theta \in [0, 1]$ through
\begin{align*}
 y_{\theta} &= u^n + \int_{t^n}^{t^n + \dt \theta}  y_f(\tau) \mathrm d \tau.
\end{align*}
For the fourth-order HB-I2DRK4-2s scheme, this results in the Butcher tableau in Tbl.~\ref{tbl:contHB-I2DRK4-2s}.
The Butcher tableau of the sixth-order HB-I2DRK6-3s scheme can be found in the
appendix, see. Tbl.~\ref{tbl:contHB-I2DRK6-3s}. More Butcher tableaux, also for
more than two derivatives or non-equidistantly spaced collocation points $c_s$,
can be generated with the MATLAB code that can be found
in our reproducibility repository \cite{ranocha2023multiderivativeRepro}.

\begin{table}[h!]
\centering
\caption{Runge-Kutta table of the continuous HB-I2DRK4-2s method. The method is uniformly fourth-order in $0 \leq \theta \leq 1$. As usual, the original Runge-Kutta method is obtained by setting $\theta = 1$. }\label{tbl:contHB-I2DRK4-2s}
\begin{tabular}{c|cc||cc}
$0$ & $0$   & $0$   & $0$   & $0$ \\
$1$ & $1/2$   & $1/2$   & $1/12$ & $-1/12$ \\ \hline
  & $\frac{\theta^4}{2} - \theta^3 + \theta$   & $-\frac{\theta^3(\theta - 2)} 2$
  & $\frac{\theta^2(3\theta^2 - 8\theta + 6)}{12}$ & $\frac{\theta^3(3\theta - 4)}{12}$
\end{tabular}
\end{table}

\begin{remark}
 It can already be mentioned at this point that for the problems we have considered in this work, the difference between using an interpolation-based output or a continuous Runge-Kutta type output is negligible, see Sec.~\ref{sec:numerical_experiments}.
\end{remark}

%% file: sec_stability.tex
Applying a multiderivative Runge-Kutta method to the scalar linear
test problem
\begin{equation}
  u'(t) = \lambda u(t), \qquad \lambda \in \C,
\end{equation}
results in the iteration
\begin{equation}
  u^{n+1} = R(z) u^{n}, \qquad z = \lambda \dt,
\end{equation}
where $R$ is the stability function of the method. Analogously to classical
Runge-Kutta methods, introducing the relaxation parameter $\gamma$ in
\begin{equation}
  u^{n+1}_\gamma = u^{n} + \gamma \left( u^{n+1} - u^{n} \right)
\end{equation}
can only increase the size of the stability domain if $\gamma \le 1$
(cf.\ Theorem~3.1 of \cite{ketcheson2019relaxation}).

\begin{theorem}
\label{thm:stability-domain-monotone}
  Consider the stability function
  \begin{equation}
    R_\gamma(z) = 1 + \gamma \left( R(z) - 1 \right)
  \end{equation}
  of a relaxed update. For $0 \le \gamma_1 \le \gamma_2$, the stability
  domain of $R_{\gamma_1}$ includes the one of $R_{\gamma_2}$.
\end{theorem}
\begin{proof}
  Let $w := R(z) - 1$. Then, $|R_{\gamma_2}(z)| = |1 + \gamma_2 w| \le 1$
  implies $|R_{\gamma_1}(z)| = |1 + \gamma_1 w| \le 1$ for
  $0 \le \gamma_1 \le \gamma_2$.
\end{proof}

\begin{corollary}
  Given $\gamma \in [0, 1]$, the stability domain of the relaxed update
  \begin{equation}
    u^{n+1}_\gamma = u^{n} + \gamma \left( u^{n+1} - u^{n} \right)
  \end{equation}
  is a superset of the one of the baseline method.
\end{corollary}

\begin{corollary}
  If the baseline method is $A$-stable, then its relaxed version is $A$-stable if $\gamma \leq 1$.
\end{corollary}

For $\gamma> 1$, we can expect the stability regions to become smaller.
For example, consider the implicit midpoint (Runge-Kutta) method
\begin{equation}
  u^{n+1} = u^{n} + \dt f\left( \frac{u^{n+1} + u^{n}}{2} \right)
\end{equation}
with stability function
\begin{equation}
  R(z) = \frac{1 + z / 2}{1 - z / 2}.
\end{equation}
Then, $z \to \infty$ yields
\begin{equation}
  R_\gamma(z)
  =
  1 + \gamma \left( R(z) - 1 \right)
  \to
  1 - 2 \gamma.
\end{equation}
Thus, $A$-stability is lost immediately for $\gamma > 1$.
For $A$-stable methods with a stability domain strictly
larger than the left half of the complex plane, we can expect to find an
upper bound on the relaxation parameter $\gamma$ ensuring $A$-stability.
A necessary condition is given in the following theorem.

\begin{theorem}
  \label{thm:L-stable-upper-bound-for-A-stability}
  For an $L$-stable method, $\gamma \le 2$ is necessary for $A$-stability of
  the relaxed update.
\end{theorem}
\begin{proof}
  For an $L$-stable method, $R(\infty) = 0$. Thus, $z \to \infty$ yields
  \begin{equation}
    |R_\gamma(z)|
    =
    |1 + \gamma \left( R(z) - 1 \right)|
    \to
    |1 - \gamma|.
  \end{equation}
  Hence, $\gamma > 2$ implies $|R_\gamma(\infty)| > 1$.
\end{proof}

The constraint $\gamma \le 2$ is not prohibitive since
$\gamma = 1 + \O(\dt^{p-1})$ for a $p$-th order baseline method.

\begin{lemma}
  The upper bound $\gamma \le 2$ of
  Theorem~\ref{thm:L-stable-upper-bound-for-A-stability} is sharp for the
  implicit Euler method.
\end{lemma}
\begin{proof}
  We have
  \begin{equation}
    R_\gamma(z)
    =
    1 + \gamma \left( R(z) - 1 \right)
    =
    1 + \gamma \left( \frac{1}{1 - z} - 1 \right)
    =
    \frac{1 - z + \gamma z}{1 - z}.
  \end{equation}
  For $\gamma = 2$, this becomes
  \begin{equation}
    R_2(z) = \frac{1 + z}{1 - z}.
  \end{equation}
  Up to a scaling of $z$ by a factor of two, this is the stability function
  of the implicit midpoint method. Monotonicity of the stability domain of
  the relaxed update (Theorem~\ref{thm:stability-domain-monotone}) proves the
  claim.
\end{proof}

Unfortunately, strict $A$-stability is lost for the relaxed update with fixed
relaxation parameter $\gamma$ for some methods already for values of
$\gamma$ close to unity. Nevertheless, $A(\alpha)$-stability is often preserved
with reasonable values $\alpha \approx 90\degree$ as shown in
Figures~\ref{fig:stability-angles-1} and
\ref{fig:stability-angles-2} for some methods.

\input{sec_stability_plots.tex}

\begin{remark}
  \label{rem:L-stability}
  Strictly speaking, $L$-stability is lost for the relaxed update if
  $\gamma \ne 1$ since
  \begin{equation}
    R_\gamma(\infty)
    =
    1 + \gamma \left( R(\infty) - 1 \right)
    =
    1 - \gamma.
  \end{equation}
  However, $1 - \gamma = \O(\dt^{p-1})$ for a $p$-th order method. Thus,
  the deviation is just small and we can still expect to get enough
  (desired) damping.
\end{remark}

\begin{remark}
  \label{rem:SSP}
  Strong stability preserving (SSP) methods preserve convex stability
  properties satisfied by the explicit Euler method
  \cite{gottlieb2011strong}. By convexity, relaxed updates of SSP methods
  are still SSP if $\gamma \in [0, 1]$. There are multiple approaches to
  extend SSP methods to multiderivative schemes, including the recent
  approach of \cite{gottlieb2022high} providing implicit SSP methods without
  time step constraints. We cannot expect to preserve the SSP property
  unconditionally for $\gamma > 1$. If relaxation is used to impose an
  inequality for a convex functional $\eta$, it can thus be a reasonable choice
  to restrict the relaxation parameter $\gamma$ to $[0, 1]$, i.e., to enable
  relaxation only if the baseline method is not already $\eta$-dissipative.
\end{remark}

%% file: sec_stability_plots.tex
\begin{figure}[htbp]
\begin{subfigure}{0.49\textwidth}
    \ifthenelse{\boolean{compilefromscratch}}
    {
      \tikzsetnextfilename{stab_HBI2DRK3}
      \input{tikz_source/stab_HBI2DRK3.tex}
    }
    {
      \includegraphics{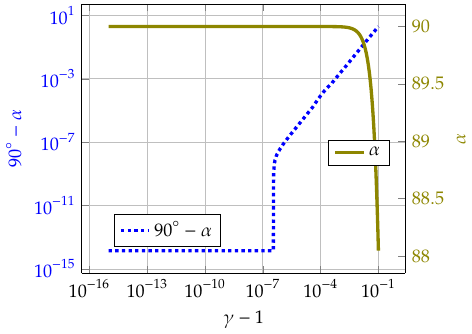}
    }
  \caption{Third-order, $L$-stable method of \cite{jaust2016implicit}.}
\end{subfigure}
\begin{subfigure}{0.49\textwidth}
    \ifthenelse{\boolean{compilefromscratch}}
    {
      \tikzsetnextfilename{stab_HBI2DRK4}
      \input{tikz_source/stab_HBI2DRK4.tex}
    }
    {
      \includegraphics{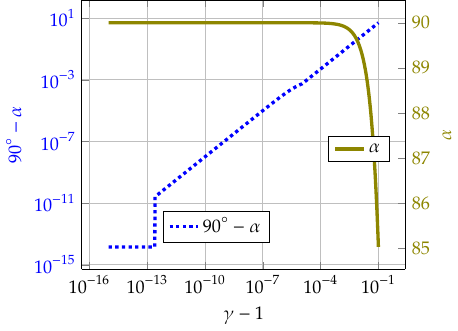}
    }
  \caption{Fourth-order, $A$-stable method of \cite[Example~1]{schutz2022parallel}.}
\end{subfigure}

\begin{subfigure}{0.49\textwidth}
    \ifthenelse{\boolean{compilefromscratch}}
    {
      \tikzsetnextfilename{stab_HBI2DRK6}
      \input{tikz_source/stab_HBI2DRK6.tex}
    }
    {
      \includegraphics{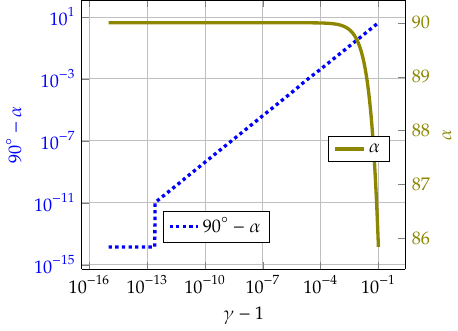}
    }
  \caption{HB-I2DRK6-3s \cite{schutz2017implicit}.}
\end{subfigure}
\begin{subfigure}{0.49\textwidth}
   \ifthenelse{\boolean{compilefromscratch}}
    {
      \tikzsetnextfilename{stab_HBI2DRK8}
      \input{tikz_source/stab_HBI2DRK8.tex}
    }
    {
      \includegraphics{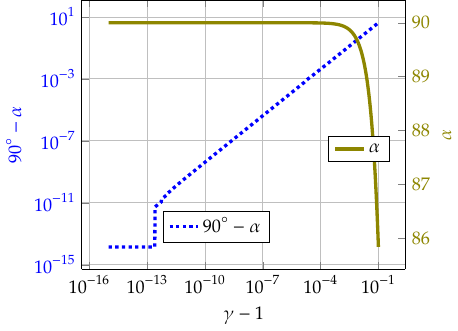}
    }
  \caption{HB-I2DRK8-4s \cite{schutz2022parallel}.}
\end{subfigure}

\begin{subfigure}{0.49\textwidth}
   \ifthenelse{\boolean{compilefromscratch}}
    {
      \tikzsetnextfilename{stab_HBI3DRK9}
      \input{tikz_source/stab_HBI3DRK9.tex}
    }
    {
      \includegraphics{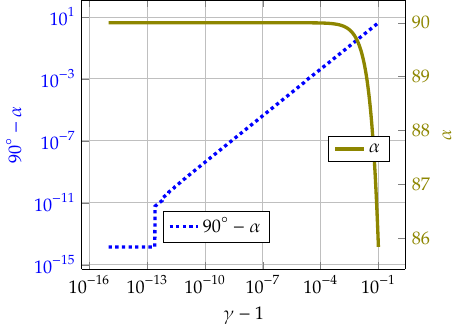}
    }
  \caption{HB-I3DRK9-3s.}
\end{subfigure}
\begin{subfigure}{0.49\textwidth}
  \ifthenelse{\boolean{compilefromscratch}}
    {
      \tikzsetnextfilename{stab_HBI3DRK12}
      \input{tikz_source/stab_HBI3DRK12.tex}
    }
    {
      \includegraphics{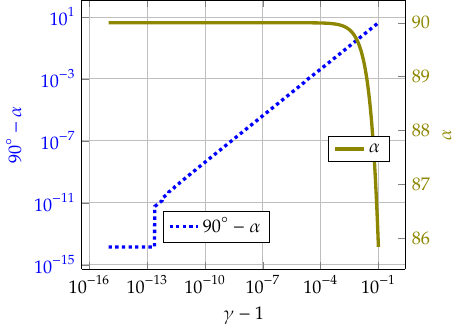}
    }
    \caption{HB-I3DRK12-4s.}
\end{subfigure}
\caption{Angles of $A(\alpha)$-stability for the relaxed update with
           fixed relaxation parameter $\gamma > 1$ for some implicit
           multiderivative methods.}
\label{fig:stability-angles-1}
\end{figure}

\begin{figure}[htb]
\begin{subfigure}{0.49\textwidth}
  \ifthenelse{\boolean{compilefromscratch}}
    {
      \tikzsetnextfilename{stab_HBI4DRK8}
      \input{tikz_source/stab_HBI4DRK8.tex}
    }
    {
      \includegraphics{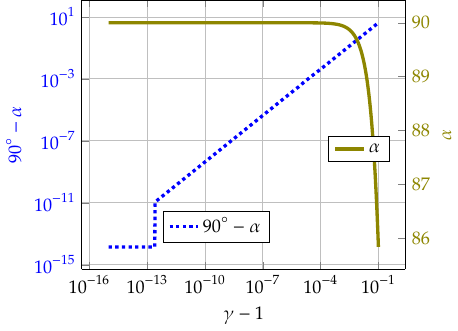}
    }
  \caption{HB-I4DRK8-2s.}
\end{subfigure}
\begin{subfigure}{0.49\textwidth}
  \ifthenelse{\boolean{compilefromscratch}}
    {
      \tikzsetnextfilename{stab_HBI4DRK12}
      \input{tikz_source/stab_HBI4DRK12.tex}
    }
    {
      \includegraphics{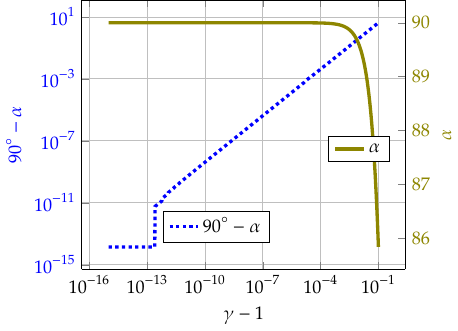}
    }
  \caption{HB-I4DRK12-3s.}
\end{subfigure}

\begin{subfigure}{0.49\textwidth}
  \ifthenelse{\boolean{compilefromscratch}}
    {
      \tikzsetnextfilename{stab_SSP3}
      \input{tikz_source/stab_SSP3.tex}
    }
    {
      \includegraphics{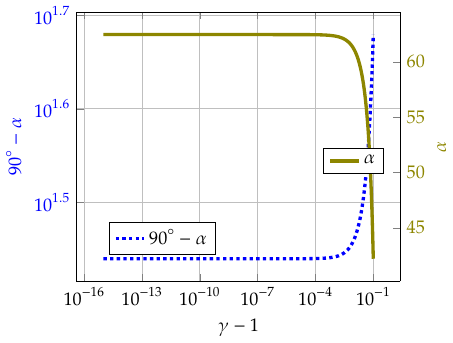}
    }
  \caption{SSP-I2DRK3-2s of \cite{gottlieb2022high}.}
\end{subfigure}
\begin{subfigure}{0.49\textwidth}
  \ifthenelse{\boolean{compilefromscratch}}
    {
      \tikzsetnextfilename{stab_SSP4}
      \input{tikz_source/stab_SSP4.tex}
    }
    {
      \includegraphics{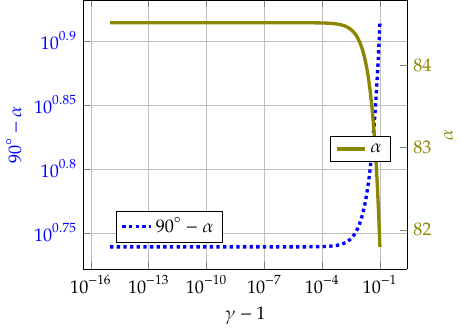}
    }
  \caption{SSP-I2DRK4-5s of \cite{gottlieb2022high}.}
\end{subfigure}
\caption{Angles of $A(\alpha)$-stability for the relaxed update with
           fixed relaxation parameter $\gamma > 1$ for some implicit
           multiderivative methods.
           }
  \label{fig:stability-angles-2}
\end{figure}

%% file: sec_numres.tex
We will demonstrate the behavior of the newly developed multiderivative
relaxation methods in this section.
We implemented some two-derivative Runge-Kutta methods in Julia
\cite{bezanson2017julia}. For them, we use automatic/algorithmic differentiation
(AD) with ForwardDiff.jl \cite{revels2016forward} to compute the second derivative
$u'' = g^{(2)}(u) = (f' f)(u)$ in a Jacobian-free manner.
We apply discontinuous Galerkin (DG) semidiscretizations of conservation
laws implemented in Trixi.jl
\cite{ranocha2022adaptive,schlottkelakemper2021purely} for simulations of
the compressible Euler equations.
We discretize the nonlinear dispersive Benjamin-Bona-Mahony equation using methods from
SummationByPartsOperators.jl \cite{ranocha2021sbp}, which are based on
FFTW.jl \cite{frigo2005design} for Fourier methods.

Please note that we use a classical method-of-lines approach for the PDE
discretizations. Thus, we only discretize the spatial operator of the PDE itself
and use AD to compute the second time derivative as $u'' = (f'f)(u)$. An
alternative would be to use a Lax-Wendroff approach, requiring the spatial
discretization of terms for both the first and the second time derivative
simultaneously \cite{lax1960systems,seal2014high,tsai2014two,babbar2022lax}.

The Julia source code to reproduce some numerical experiments is available
in our reproducibility repository \cite{ranocha2023multiderivativeRepro}.
The remaining numerical experiments are implemented in MATLAB \cite{matlabR2022b}
based on a closed-source inhouse research code developed at Hasselt University
using symbolic tools to compute higher derivatives of $u$.

The numerical experiments are roughly ordered by increasing complexity of
the ODE systems. First, we verify the convergence properties for a
nonlinear oscillator with quadratic entropy functional and two ODEs
with non-quadratic entropies. Afterwards, we consider PDE discretizations:
the 3D compressible Euler equations of fluid dynamics to demonstrate the
robustness of our approach as well as two nonlinear dispersive wave equations
where the conservation of functionals is crucial to obtain better qualitative
and quantitative results.

\subsection{Convergence experiment I: a nonlinear oscillator}
\input{sec_numres_oscillator.tex}

\subsection{Convergence experiment II: Kepler's problem}
\input{sec_numres_kepler.tex}

\subsection{Convergence experiment III: dissipated exponential entropy}

\input{sec_numres_dissipatedexponential.tex}

\subsection{3D inviscid Taylor-Green vortex}
\input{sec_numres_tgv.tex}

\subsection{Benjamin-Bona-Mahony equation}
\input{sec_numres_bbm.tex}

\subsection{Korteweg-de Vries equation}
\input{sec_numres_kdv.tex}

%% file: sec_numres_oscillator.tex
First, we demonstrate the convergence properties of the multiderivative
methods with and without relaxation. We choose the entropy-conservative
nonlinear oscillator
\begin{equation}
\label{eq:nonlinear_osc}
  u'(t) = \frac{1}{\| u(t) \|^2}
  \begin{pmatrix} -u_2(t) \\ u_1(t) \end{pmatrix},
  \text{ for } t \in (0, T),
  \qquad
  u(0) = \begin{pmatrix} 1 \\ 0 \end{pmatrix},
\end{equation}
of \cite{ranocha2021strong,ranocha2020energy,ketcheson2022computing} and
its dissipative counterpart
\begin{equation}
\label{eq:damped_nonlinear_osc}
  u'(t) = \frac{1}{\| u(t) \|^2}
  \begin{pmatrix} -u_2(t) \\ u_1(t) \end{pmatrix}
  - \epsilon \begin{pmatrix} u_1(t) \\ u_2(t) \end{pmatrix},
  \text{ for } t \in (0, T),
  \qquad
  u(0) = \begin{pmatrix} 1 \\ 0 \end{pmatrix},
\end{equation}
with analytical solution
\begin{equation}
  u(t) = r(t) \begin{pmatrix} \cos \phi(t) \\ \sin \phi(t) \end{pmatrix},
  \qquad
  r(t) = \exp(-\epsilon t),
  \;
  \phi(t) =
  \begin{cases}
    t, & \epsilon = 0, \\
    \frac{\exp(2 \epsilon t) - 1}{2 \epsilon}, & \epsilon \ne 0.
  \end{cases}
\end{equation}
For the entropy-dissipative nonlinear oscillator, we choose the damping
coefficient $\epsilon = 10^{-2}$.
In both cases, we choose the entropy $\eta(u) = \| u \|^2$ with the
standard Euclidean inner product norm $\| \cdot \|$.

\begin{figure}[ht]
 \centering
     \ifthenelse{\boolean{compilefromscratch}}
    {
      \tikzsetnextfilename{osc_cons_ct}
      \input{tikz_source/osc_cons_ct.tex}
    }
    {
      \includegraphics{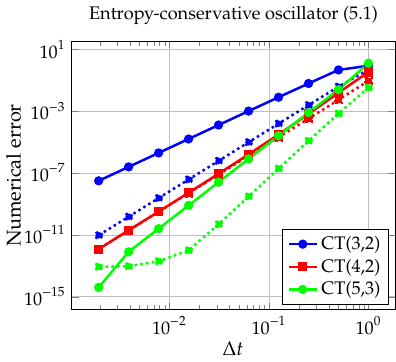}
    }
\hspace{1cm}
     \ifthenelse{\boolean{compilefromscratch}}
    {
      \tikzsetnextfilename{osc_diss_ct}
      \input{tikz_source/osc_diss_ct.tex}
    }
    {
      \includegraphics{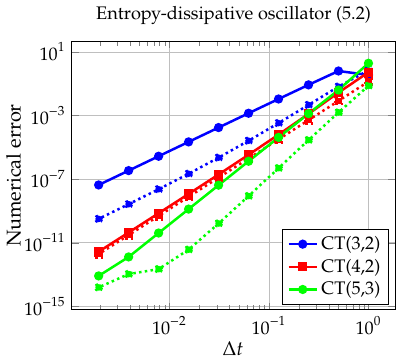}
    }
\caption{
    Convergence results for the nonlinear oscillators using the explicit third, fourth and fifth order methods of \cite{chan2010explicit}, see also Sec.~\ref{sec:ctbutcher}.
    Dotted lines correspond to the relaxed schemes, straight lines to the baseline schemes. }
  \label{fig:convergence-CT}
\end{figure}

The errors at the final time $T = 10$ of the three explicit two-derivative methods CT(3,2), CT(4,2) and CT(5,3) from \cite{chan2010explicit}, see also Sec.~\ref{sec:ctbutcher}, are shown in Figure~\ref{fig:convergence-CT}.
The entropy estimation $\eta^{\text{new}}$ has been computed using the four-point Gauss-Lobatto rule combined with the quintic polynomial outlined in \eqref{eq:entropyestimation}. Note that the CT-methods are not collocation-based, so it is not obvious how to choose a continuous output.

We can clearly observe the expected orders of convergence of 3, 4, and 5 for the baseline schemes (straight lines).
For the entropy-conservative case, the odd order methods show an increase in order. This phenomenon is well-known, and has been analyzed in \cite{ranocha2020relaxationHamiltonian}.
In any case, also for the entropy-dissipative one, the relaxed methods of odd order behave significantly better than the baseline methods in that the error coefficients seem to be significantly smaller. For the fourth-order scheme, this difference is less pronounced; it is only visible for large timesteps.

Having considered explicit schemes, we now turn our attention to the class of implicit schemes. In particular, we consider some collocation-based schemes, see Sec.~\ref{sec:hbbutcher}, and a third-order SSP scheme (SSP-I2DRK3-2s) recently introduced in \cite{gottlieb2022high}, see also Eq.~\eqref{eq:SSP3}. Numerical results have been reported in Fig.~\ref{fig:implicit_oscillator}.
Again, design orders of accuracy are visible for the baseline schemes in all cases. The third-order SSP scheme profits a lot from relaxation, which was to be expected, as the order is increased by one. For the other two methods, the difference between relaxed and baseline scheme is less pronounced.

For the entropy-dissipative case and the collocation-based schemes, we have computed $\eta^{\text{new}}$ in two different fashions as explained in Sec.~\ref{sec:entropy-estimates}. An estimate is obtained through an appropriate Gauss-Lobatto quadrature combined with a \emph{polynomial interpolation} (dotted) or a \emph{continuous output} (dashed). For the sixth-order scheme, a Gauss-Lobatto quadrature with five points is used to obtain the appropriate order. As the SSP scheme does not come with a continuous output, there, only an estimate via interpolation was possible to obtain.
It is very difficult to distinguish the dotted and the dashed lines in Fig.~\ref{fig:implicit_oscillator} (right). While they are not the same, they lie practically on top of each other, from which one can conclude that the results are not very sensitive to the choice of entropy estimator.

\begin{figure} [ht]
    \centering
 \ifthenelse{\boolean{compilefromscratch}}
    {
      \tikzsetnextfilename{osc_cons_impl}
      \input{tikz_source/osc_cons_impl.tex}
    }
    {
      \includegraphics{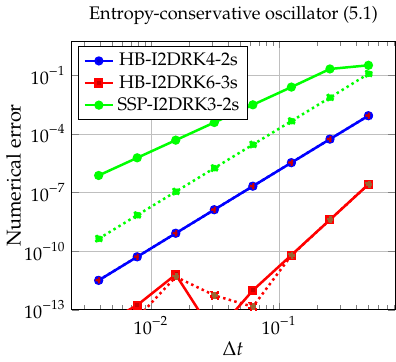}
    }
 \ifthenelse{\boolean{compilefromscratch}}
    {
      \tikzsetnextfilename{osc_diss_impl}
      \input{tikz_source/osc_diss_impl.tex}
    }
    {
      \includegraphics{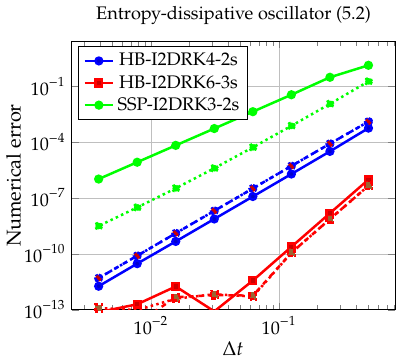}
    }
\caption{
    {Convergence results for the nonlinear oscillators using several implicit schemes, see also Sec.~\ref{sec:sspbutcher} and Sec.~\ref{sec:hbbutcher}. Straight lines correspond to the baseline schemes, while dotted and dashed lines use relaxation.
    For the entropy-conservative case, only dotted lines are plotted, as $\eta^{\text{new}} \equiv \eta(u^n)$ is trivial to obtain. For the entropy-dissipative case, estimating $\eta^{\text{new}}$ is an important part of the algorithm, see Sec.~\ref{sec:entropy-estimates}.
    Dotted means that the estimate of $\eta^{\text{new}}$ has been obtained using a Hermite-Birkhoff interpolation of $u$, while dashed means that the estimate has been obtained using the continuous Runge-Kutta output. Please note that the latter is not possible for the SSP scheme.}\label{fig:implicit_oscillator}
}
\end{figure}

\paragraph{Error growth in time for nonlinear oscillators}

	\begin{figure} [ht]
        \centering
     \ifthenelse{\boolean{compilefromscratch}}
    {
      \tikzsetnextfilename{osc_cons_errorgrowth_ct}
      \input{tikz_source/osc_cons_errorgrowth_ct.tex}
      \tikzsetnextfilename{osc_cons_entropy_ct}
      \input{tikz_source/osc_cons_entropy_ct.tex}

    }
    {
      \includegraphics{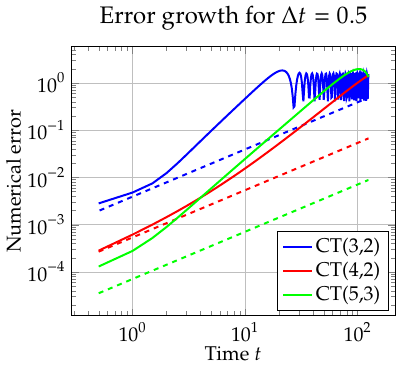}
      \includegraphics{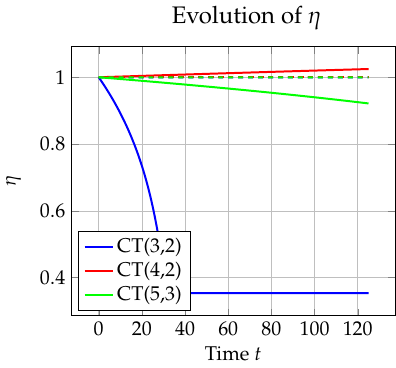}
    }

     \ifthenelse{\boolean{compilefromscratch}}
    {
      \tikzsetnextfilename{osc_diss_errorgrowth_ct}
      \input{tikz_source/osc_diss_errorgrowth_ct.tex}
      \tikzsetnextfilename{osc_diss_entropy_ct}
      \input{tikz_source/osc_diss_entropy_ct.tex}

    }
    {
      \includegraphics{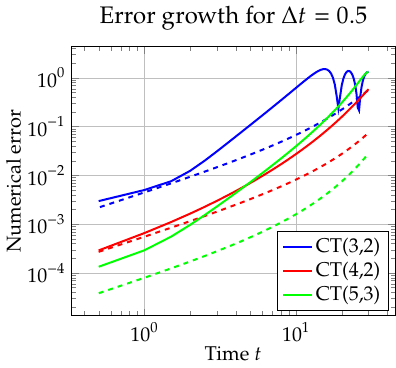}
      \includegraphics{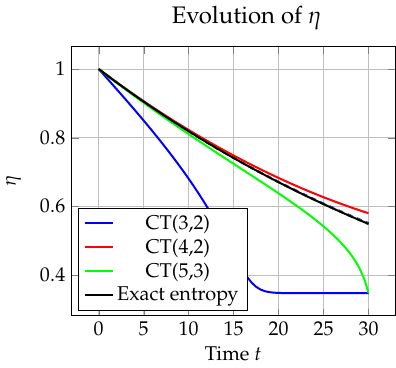}
    }
        \caption{Temporal evolution of the numerical error and the functional $\eta$ for the nonlinear oscillators using the explicit third, fourth and fifth order methods of \cite{chan2010explicit}, see also Sec.~\ref{sec:ctbutcher}. Top: conservative case \eqref{eq:nonlinear_osc}, bottom: dissipative case \eqref{eq:damped_nonlinear_osc}.
        Dashed lines denote schemes that use relaxation.
        }\label{fig:nonlinear_osc-CT}
    \end{figure}

Here, we take a closer look at the behavior of the error of the three selected multiderivative Runge-Kutta methods by Chan and Tsai \cite{chan2010explicit} for the
nonlinear oscillators described in the previous subsection.
The final time is set to $T=125$ (entropy-conservative) and $T=30$ (entropy-dissipative), and the rather large timestep of $\dt = 0.5$ is chosen.
The results for the entropy-conservative nonlinear oscillator \eqref{eq:nonlinear_osc}
are shown in Figure~\ref{fig:nonlinear_osc-CT} on the top. Clearly, the
relaxation approach conserves the entropy (up to machine accuracy)
while the baseline scheme results in an either monotonically increasing or monotonically decreasing entropy.
The error of the relaxation method grows linearly in time while the
baseline scheme results in a quadratic asymptotic error growth in time, at least until the errors are so large that the numerical solution turns out to be useless. This is something that we can only see for the baseline scheme, relaxation clearly improves the long-time behaviour of the error.

The different behavior of the error growth in time can be understood, e.g.,
using the theory of \cite{cano1997error,duran1998numerical,calvo2011error}.
Indeed, we are looking for a periodic orbit, which is a relative equilibrium
solution of the Hamiltonian system \eqref{eq:nonlinear_osc}. Thus,
structure-preserving numerical methods are expected to result not only
in improved qualitative behavior but also yield quantitatively better
results due to a reduced error growth rate in time.

The corresponding results for the entropy-dissipative nonlinear oscillator
\eqref{eq:damped_nonlinear_osc} are shown in
Figure~\ref{fig:nonlinear_osc-CT} in the lower part.
As expected for an accurate method and a strictly dissipative problem,
the entropy decays over time. It is interesting to see that the relaxed
method results in an entropy that is visually indistinguishable from the
entropy of the analytical solution while the baseline methods clearly deviate. In addition, the error of the relaxed solution is
smaller and grows slightly slower in time than the error of the baseline
method. Since this problem is entropy-dissipative, it does not fit into
the framework of \cite{cano1997error,duran1998numerical,calvo2011error}
mentioned earlier. Nevertheless, applying relaxation still improves both
the qualitative and quantitative behavior of the numerical solution.

The same investigation has been done for the implicit schemes mentioned in the previous convergence studies, see Fig.~\ref{fig:nonlinear_osc-implicit}. Similar observations can be made. It has to be noted that, again, the third-order SSP scheme profits a lot from relaxation, which is visible both from the error growth and the evolution of $\eta$. For the even-order collocation schemes, this phenomenon is less pronounced. In particular, for these schemes, the quadratic functional $\eta$, in the entropy-conservative case, is conserved already for the baseline scheme.
Obviously, this is also backed up by the findings from the convergence study in Fig.~\ref{fig:implicit_oscillator}. A (small) advantage can be seen for the collocation schemes in the entropy-dissipative case.

For the entropy-dissipative case, we estimated $\eta^{\text{new}}$ only through a Hermite-Birkhoff interpolation, as previous investigation has shown that the effect is negligible.

\begin{figure} [ht]
        \centering
 \ifthenelse{\boolean{compilefromscratch}}
    {
      \tikzsetnextfilename{osc_cons_errorgrowth_impl}
      \input{tikz_source/osc_cons_errorgrowth_impl.tex}
      \tikzsetnextfilename{osc_cons_entropy_impl}
      \input{tikz_source/osc_cons_entropy_impl.tex}

    }
    {
      \includegraphics{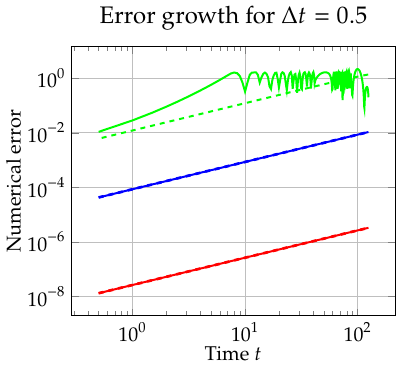}
      \includegraphics{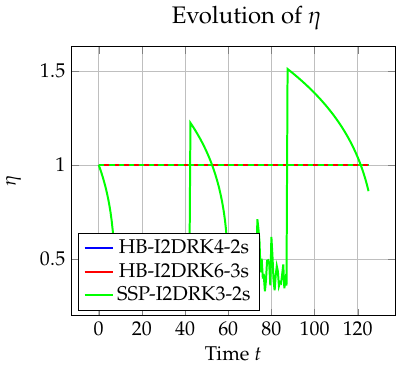}
    }

 \ifthenelse{\boolean{compilefromscratch}}
    {
      \tikzsetnextfilename{osc_diss_errorgrowth_impl}
      \input{tikz_source/osc_diss_errorgrowth_impl.tex}
      \tikzsetnextfilename{osc_diss_entropy_impl}
      \input{tikz_source/osc_diss_entropy_impl.tex}

    }
    {
      \includegraphics{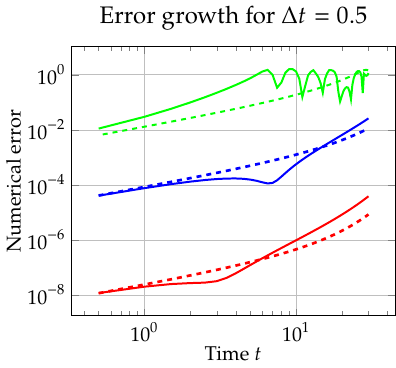}
      \includegraphics{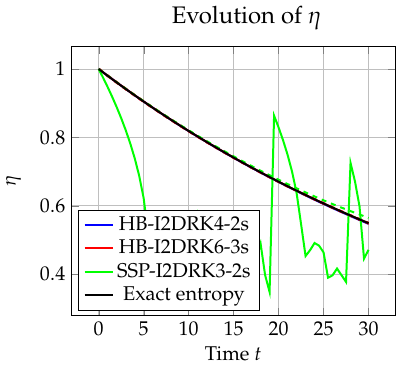}
    }
        \caption{Temporal evolution of the numerical error and the functional $\eta$ for the nonlinear oscillators using several implicit schemes, see also Sec.~\ref{sec:sspbutcher} and Sec.~\ref{sec:hbbutcher}.
        Straight lines correspond to the baseline schemes, while dashed lines use relaxation.
        Top: conservative case \eqref{eq:nonlinear_osc}, bottom: dissipative case \eqref{eq:damped_nonlinear_osc}.
        As the actual way how $\eta^{\text{new}}$ is computed for the entropy-dissipative case turned out to have a negligible effect, only the Hermite-Birkhoff interpolation of $u$ is used.
        Note that the baseline schemes for HB-I2DRK4-2s and HB-I2DRK6-3s are already excellent at preserving the functional for the entropy-conservative case, which is why it seems like there is only one plotted line. This is also very much in line with the convergence results from Fig.~\ref{fig:implicit_oscillator}.
        The legend on the right is also valid for the figure on the left.
        }\label{fig:nonlinear_osc-implicit}
    \end{figure}

%% file: sec_numres_kepler.tex
As a second test problem with a non-quadratic entropy, we discuss Kepler's problem following \cite{ranocha2020general}, which is given by
\begin{align}
 \label{eq:Kepler}
 q_i'(t) = p_i(t), \qquad p_i'(t) = - \frac{q_i(t)}{\|q(t)\|^3}, \quad i \in \{1,2\}.
\end{align}
Here, $q\colon \R^+ \rightarrow \R^2$ and $p\colon \R^+ \rightarrow \R^2$ are the unknown functions.
We choose the same initial conditions as in \cite{ranocha2020general}, i.e.,
\begin{align*}
 q(0) = \left(1-e, 0\right)^T, \qquad p(0) = \left(0, \sqrt{\frac{1-e}{1+e}}\right)^T,
\end{align*}
with eccentricity $e = \frac 1 2$. As ``entropy'' functional $\eta$, we choose the angular momentum
\begin{align*}
 \eta(q,p) := q_1 p_2 - q_2 p_1,
\end{align*}
which, for the exact solution, is a conserved quantity.

As in the previous section, we start with the class of explicit schemes from \cite{chan2010explicit}; and we add two third-derivative schemes from \cite{TurTur2017} to demonstrate that our approach is independent of the number of derivatives. Numerical results have been plotted in Fig.~\ref{fig:KeplerCT} for a final time of $T=5$. As $\eta$ is not a quadratic functional, we do not expect superconvergence and, indeed, both baseline and relaxed methods come with their design order of accuracy. However, we can repeat our observation for the odd-order schemes from the previous section, which is that the errors are typically lower, sometimes even significantly lower, and that the asymptotic regime is reached faster.

\begin{figure}[htbp]
     \centering
      \ifthenelse{\boolean{compilefromscratch}}
    {
      \tikzsetnextfilename{kepler_cons_explicit}
      \input{tikz_source/kepler_cons_explicit.tex}
    }
    {
      \includegraphics{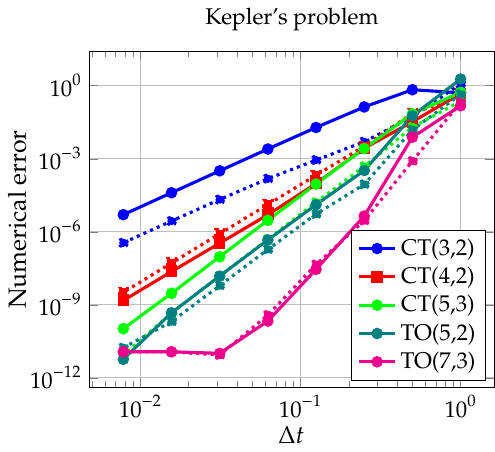}
    }
    \caption{Convergence results for Kepler's problem \eqref{eq:Kepler} using the explicit third, fourth and fifth order methods of \cite{chan2010explicit}, see also Sec.~\ref{sec:ctbutcher}; and the three-derivative schemes of order 5 and 7 of Turaci and \"{O}zi\c{s} \cite{TurTur2017}, see also Sec.~\ref{sec:turturbutcher}.
    Dotted lines correspond to the relaxed schemes, straight lines to the baseline schemes. }
    \label{fig:KeplerCT}
\end{figure}

Fig.~\ref{fig:Kepler_Convg_Implicit} shows convergence results for a couple of implicit collocation-based schemes of orders three, eight, and twelve, with up to four derivatives. Due to the relatively high order of the schemes involved, the final time has been increased to $T=50$.
On the left, baseline schemes are plotted, while on the right, the relaxed versions are displayed.
Design accuracies are met, with the exception of the baseline scheme of order three that has not reached the asymptotic regime yet. (It has been checked that it converges, though.) Again, the relaxed version seems to reach the asymptotic regime way faster.
Please note that as $\eta$ is not quadratic any more, the functional is not preserved exactly by the baseline schemes, but it is by the relaxed schemes.

\begin{figure}[htbp]
        \centering

 \ifthenelse{\boolean{compilefromscratch}}
    {
      \tikzsetnextfilename{kepler_conv_implicit_base}
      \input{tikz_source/kepler_conv_implicit_base.tex}
      \tikzsetnextfilename{kepler_conv_implicit_relax}
      \input{tikz_source/kepler_conv_implicit_relax.tex}

    }
    {
      \includegraphics{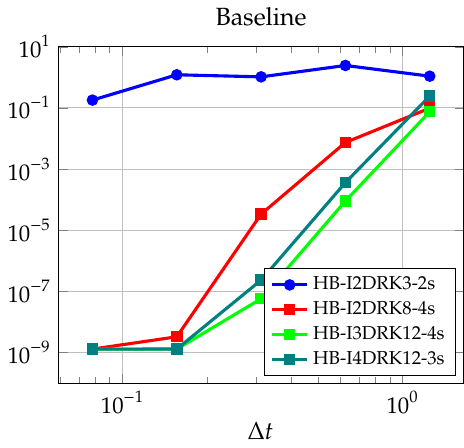}
      \includegraphics{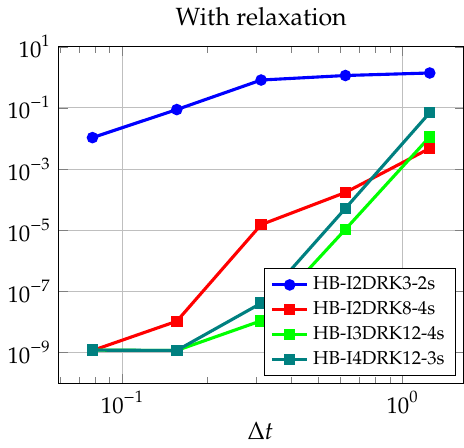}
    }
        \caption{Several multiderivative timesteppers applied to the entropy conserving Kepler problem for various values of $k_{\max}$. The final end time has been set to $T = 50$ to account for the rather high order of the schemes involved (three, eight, twelve and twelve, respectively).
        The left figure displays the baseline schemes, while the right figure displays the schemes with a relaxation.
        Convergence stalls relatively early due to the used, numerically computed, reference solution, which obviously tends to become less accurate with larger final time $T$.
        }\label{fig:Kepler_Convg_Implicit}
    \end{figure}

\paragraph{Error growth in time for Kepler's problem}

Similar as before, we show the behaviour of the error and the functional $\eta$ as a function of time, exemplarily for the explicit schemes of Chan and Tsai \cite{chan2010explicit}, see also Sec.~\ref{sec:ctbutcher}, and Turaci and \"{O}zi\c{s} \cite{TurTur2017}, see also Sec.~\ref{sec:turturbutcher}. Findings have been plotted in Fig.~\ref{fig:kepler_errorgrowth_ct}. On the left, error evolution over time and on the right, functional evolution over time have been plotted. In all the cases, the baseline scheme behaves worse than its relaxed counterpart.

	\begin{figure}[htbp]
        \centering
 \ifthenelse{\boolean{compilefromscratch}}
    {
      \tikzsetnextfilename{kepler_errorgrowth_expl}
      \input{tikz_source/kepler_errorgrowth_expl.tex}
      \tikzsetnextfilename{kepler_entropy_expl}
      \input{tikz_source/kepler_entropy_expl}

    }
    {
      \includegraphics{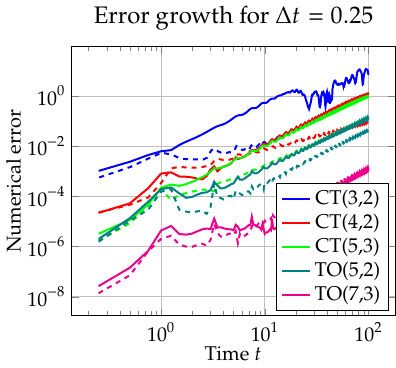}
      \includegraphics{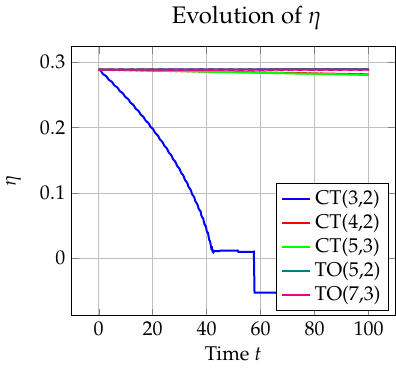}
    }

        \caption{Temporal evolution of the numerical error and the functional $\eta$ for Kepler's problem \eqref{eq:Kepler} using the explicit third, fourth and fifth order methods of \cite{chan2010explicit}, see also Sec.~\ref{sec:ctbutcher}; and the three-derivative schemes of order 5 and 7 of Turaci and \"{O}zi\c{s} \cite{TurTur2017}, see also Sec.~\ref{sec:turturbutcher}.
        Dotted lines correspond to the relaxed schemes, straight lines to the baseline schemes.
        }\label{fig:kepler_errorgrowth_ct}
    \end{figure}

%% file: sec_numres_dissipatedexponential.tex
As a final ODE problem, we consider an admittedly rather artificial equation with a non-quadratic dissipative entropy. The problem is taken from \cite{ranocha2020general}, it is given by the equation
\begin{align}
 \label{eq:dissexp}
 u'(t) = -e^{u(t)}, \qquad u(0) = 0.5.
\end{align}
For all results, final time has been set to $T=2.5$.

For the nonlinear functional $\eta(u) = e^u$, it is straightforward to check that solutions to \eqref{eq:dissexp} fulfill $\frac{\mathrm{d}}{\mathrm{d}t} \eta(u) \leq 0$.

In Fig.~\ref{fig:DissExponentialConvg}, convergence results for a variety of methods have been plotted. First of all, the findings from the previous sections are valid for this testcase as well. As we are dealing with a dissipative rather than a conservative case, it is to be expected that the quality of the solution does not improve dramatically by introducing relaxation, which is also visible from Fig.~\ref{fig:DissExponentialConvg}. In fact (not plotted here), already the baseline schemes produce, in the range of $\dt$ considered here, solutions with a decaying entropy $\eta$. The advantage of using a relaxed scheme is that one can be absolutely certain, in a mathematically rigorous way, that the functional decays. Please note that, in contrast to previous findings, the relaxed schemes have a slightly higher error constant here. Also for this case, continuous output versus Hermite-Birkhoff interpolation in estimating $\eta^{\text{new}}$ has been tested; the differences are not significant and hence not shown here.

\begin{figure}[h!]
     \centering

\ifthenelse{\boolean{compilefromscratch}}
    {
      \tikzsetnextfilename{de_conv_ct}
      \input{tikz_source/de_conv_ct.tex}
      \tikzsetnextfilename{de_conv_impl}
      \input{tikz_source/de_conv_impl}

    }
    {
      \includegraphics{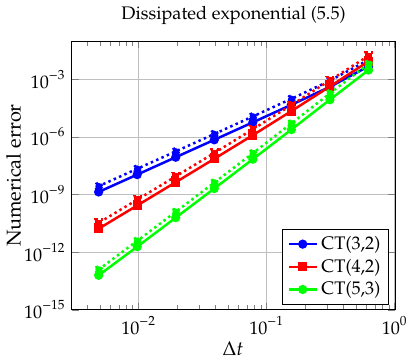}
      \includegraphics{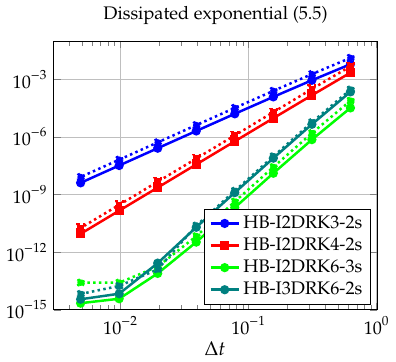}
    }
    \caption{Convergence results for the dissipated exponential problem \eqref{eq:dissexp}. Left: explicit third, fourth and fifth order methods of \cite{chan2010explicit}, see also Sec.~\ref{sec:ctbutcher}; and the three-derivative schemes of order 5 and 7 of Turaci and \"{O}zi\c{s} \cite{TurTur2017}, see also Sec.~\ref{sec:turturbutcher}, are used. Right: some implicit collocation-based schemes.
    Dotted lines correspond to the relaxed schemes, straight lines to the baseline schemes. }
    \label{fig:DissExponentialConvg}
\end{figure}

%% file: sec_numres_tgv.tex
Next, we demonstrate the robustness of the methods for a 3D simulation
of compressible turbulence. For this, we consider the 3D compressible Euler
equations of an ideal gas with ratio of specific heats $\gamma = 1.4$
and choose the classical inviscid Taylor-Green vortex initial data
\cite{gassner2016split}
\begin{equation}
\begin{gathered}
  \rho = 1, \;
  v_1 =  \sin(x_1) \cos(x_2) \cos(x_3), \;
  v_2 = -\cos(x_1) \sin(x_2) \cos(x_3), \;
  v_3  = 0, \\
  p = \frac{\rho^0}{\mathrm{Ma}^2 \gamma} + \rho^0 \frac{\cos(2 x_1) \cos(2 x_3) + 2 \cos(2 x_2) + 2 \cos(2 x_1) + \cos(2 x_2) \cos(2 x_3)}{16},
\end{gathered}
\end{equation}
with Mach number $\mathrm{Ma} = 0.1$, density $\rho$, velocity $v$, and
pressure $p$. The domain $[-\pi, \pi]^3$ is equipped with periodic boundary
conditions. The spatial semidiscretization is performed using flux differencing
DG methods, specifically the discontinuous Galerkin
spectral element method (DGSEM) on tensor product elements using the
Gauss-Lobatto-Legendre nodes with $8^3$ uniform elements and polynomials of
degree $p = 3$, resulting in \num{32768} degrees of freedom (DOFs). We use the
entropy-conservative flux of
\cite{ranocha2018comparison,ranocha2020entropy,ranocha2021preventing}
in the volume; the surface flux is either the same entropy-conservative flux
or a dissipative local Lax-Friedrichs/Rusanov numerical flux.
See \cite{fisher2013high,gassner2016split,ranocha2023efficient}
for a description of these methods.

\begin{figure}[htbp]
\centering
  \begin{subfigure}{0.49\textwidth}
  \centering

\ifthenelse{\boolean{compilefromscratch}}
    {
      \tikzsetnextfilename{tgv_cons}
      \input{tikz_source/tgv_cons.tex}
    }
    {
      \includegraphics{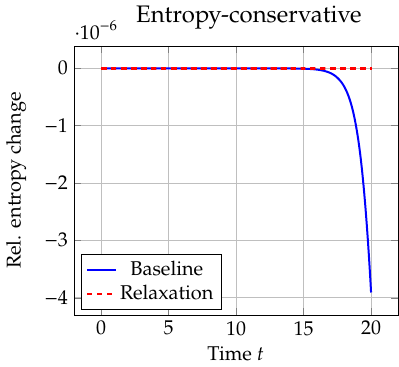}
    }
    \caption{Entropy-conservative semidiscretization.}
  \end{subfigure}%
  \hspace{\fill}
  \begin{subfigure}{0.49\textwidth}
  \centering
\ifthenelse{\boolean{compilefromscratch}}
    {
      \tikzsetnextfilename{tgv_diss}
      \input{tikz_source/tgv_diss.tex}
    }
    {
      \includegraphics{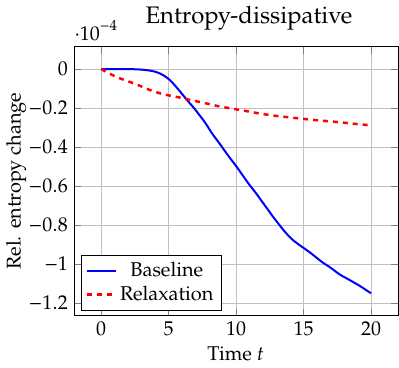}
    }
    \caption{Entropy-dissipative semidiscretization.}
  \end{subfigure}%
  \caption{Numerical results for the 3D inviscid Taylor-Green vortex using
           discontinuous Galerkin methods in space and the explicit fourth-order,
           two-derivative method of \cite{chan2010explicit} in time. Shown is the
           relative change of the entropy
           $\bigl(\eta(u) - \eta(u^0) \bigr) / |\eta(u^0)|$.
           }
  \label{fig:taylor_green_vortex}
\end{figure}

We choose the classical physical entropy of the compressible Euler equations
given by
\begin{equation}
  \int -\frac{\rho s}{\gamma - 1},
  \qquad
  s = \log \frac{p}{\rho^\gamma},
\end{equation}
and its numerical realization via the collocation quadrature rule associated
with the DGSEM.
The entropy of numerical solutions obtained by the fourth-order, two-derivative
method of \cite{chan2010explicit} with time step size $\dt = 0.005$ is shown in
Figure~\ref{fig:taylor_green_vortex}\footnote{Since the time step size of explicit
methods is limited by stability for this problem, an adaptive choice
of the step sizes typically leads to an asymptotically constant $\dt$
\cite{ranocha2023stability}.}.
Clearly, the baseline method yields a visible
change of the entropy over time in the entropy-conservative case while relaxation
is able to conserve the entropy as expected.
In the entropy-dissipative case, the baseline time integration method yields
less dissipation initially but more dissipation for larger times.
This numerical experiment demonstrates the robustness of the method even when
combined with highly-nonlinear discretizations of PDEs and complicated entropy
functionals.

%% file: sec_numres_bbm.tex
Next, we consider solitary wave solutions of the Benjamin-Bona-Mahony (BBM)
equation \cite{benjamin1972model}
\begin{equation}
\label{eq:bbm}
  \partial_t u + \partial_x u + \partial_x \frac{u^2}{2} - \partial_x^2 \partial_t u = 0.
\end{equation}
We choose the spatial domain $[-90, 90]$ equipped with periodic boundary
conditions and an initial condition corresponding to the solitary wave
solution
\begin{equation}
  u(t, x) = \frac{A}{\cosh\bigl(K (x - c t) \bigr)^2},
  \quad
  \text{with }
  A = 3 (c - 1),
  \;
  K = \frac{1}{2} \sqrt{1 - 1 / c},
  \;
  c = 1.2.
\end{equation}
We use the method-of-lines approach and combine an entropy-conservative
semidiscretization of \cite{ranocha2021broad} with the two-derivative
Runge-Kutta method \eqref{eq:CT(4,2)} of \cite{chan2010explicit}.
The schemes conserve a discrete equivalent of the invariant
\begin{equation}
  \int \left( u(t, x)^2 + \bigl( \partial_x u(t, x) \bigr)^2 \right) \dif x.
\end{equation}

\begin{figure}[htbp]
    \centering
    \ifthenelse{\boolean{compilefromscratch}}
    {
      \tikzsetnextfilename{bbm_errorgrowth}
      \input{tikz_source/bbm_errorgrowth.tex}
      \tikzsetnextfilename{bbm_entropy}
      \input{tikz_source/bbm_entropy.tex}
    }
    {
      \includegraphics{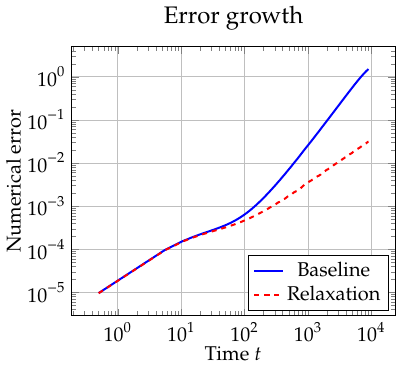}
      \includegraphics{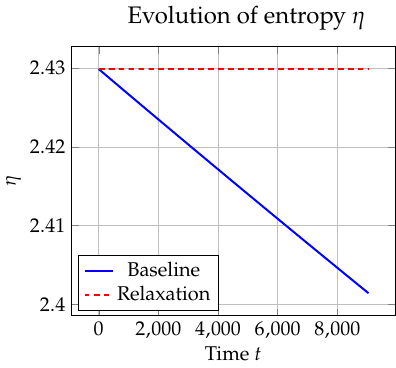}
    }

  \caption{Numerical results for the BBM equation \eqref{eq:bbm} discretized
           in space using an entropy-conservative Fourier collocation method
           with $2^8$ nodes \cite{ranocha2021broad} and using the
           fourth-order method CT(4,2) of \cite{chan2010explicit} with
           $\dt = 0.5$ in time.}
  \label{fig:bbm-CT(4,2)}
\end{figure}

The numerical results are shown in Figure~\ref{fig:bbm-CT(4,2)}. Clearly,
the relaxation method conserves the entropy while the baseline scheme is
entropy-dissipative in this case. The asymptotic error growth is linear
for the relaxed method and quadratic for the baseline scheme. This behavior
is expected based on the theory of \cite{araujo2001error}, which is an
extension of the relative equilibrium theory of \cite{duran1998numerical}
to Hamiltonian PDEs. Such a property has first been proved rigorously for
the Korteweg-de Vries equation in \cite{frutos1997accuracy} and later been
extended to other equations such as the nonlinear Schrödinger equation
\cite{duran2000numerical}. The general framework also covers many other
nonlinear dispersive wave equations \cite{ranocha2021rate}.

%% file: sec_numres_kdv.tex
Here, we consider the Korteweg-de Vries-equation, which is a third-order nonlinear partial differential equation, given by
\begin{alignat}{2}
 \label{eq:kdv}
 u_t + \left(\frac{u^2}{2}\right)_x + u_{xxx} &= 0, &\qquad& (x, t) \in \R\times \R^+, \\
 u(0, x) &= u_0(x), &\qquad& x \in \R.
\end{alignat}
Please note that this is, due to the occurrence of both nonlinear and third-order terms, a rather complex testcase. In particular, the term $u_{xxx}$ has eigenvalues on the imaginary axis.
There exists an exact soliton solution to this equation \cite{frutos1997accuracy}, given by
\begin{align*}
 u(t, x) = 2 \cosh\left(\frac{\tilde x}{\sqrt 6} \right)^{-2},
\end{align*}
where we have defined $\tilde x$ to be
\begin{align*}
 \tilde x = \mmod\biggl(x - \frac{2t}{3}, 80\biggr) - 40.
\end{align*}
The solution $u$ on $\R$ is hence a periodic repetition of the solution at, e.g., $x \in [0, 80]$, which is the domain that we actually discretize, with periodic boundary conditions, in our work. The time period is given by $T_{\text{per}} = \frac{3 \cdot 80}{2} = 120$. The discretization is done through a 17-point stencil Lagrangian finite difference scheme on a grid with 1000 elements, rendering the spatial resolution relatively high. It is evident both from the character of the Korteweg-de-Vries equation and the high-order discretization that implicit schemes are an absolute necessity here.
For the spatial discretization, the Burgers-term $(u^2/2)_x$ in \eqref{eq:kdv} has been rewritten in the equivalent (at least for smooth solutions) split-form
\begin{align*}
 \frac 1 2 (u^2)_x \equiv \frac 1 3 (u^2)_x + \frac 1 3 u u_x.
\end{align*}
This form has been used as with this form, it is straightforward to show that the functional
\begin{align}
 \label{eq:kdventropie}
 \eta(u) = \int_{\R} u^2 \mathrm{d} x
\end{align}
is preserved.

Fig.~\ref{fig:KdVConvg} shows convergence results for three implicit schemes, always as baseline and relaxation scheme. The final time has been set as a rather short $T=10$ on the left, and $T=360$ on the right. $T=360$ corresponds to three temporal periods of the solution. As before, only little difference can be seen for the HB-I2DRK4-2s and the HB-I2DRK6-3s scheme in their baseline and relaxed version. However, the behavior of the HB-I2DRK3-2s scheme is very interesting. For the smaller time $T=10$, the relaxed version converges with third order, while for the larger time $T=360$, it converges with fourth order rather than three, which is the baseline scheme's order. Obviously, this is a very odd behavior that one would not have anticipated. Hence, in Fig.~\ref{fig:kdv_errorgrowth}, we have plotted for this particular scheme the error as a function of time for the baseline scheme (left) and the relaxed scheme (right) for various values of $\dt$. The error distribution for the baseline scheme is relatively smooth, and from about $t = 10$ on, it behaves quadratically in time. This is not true for the relaxed scheme, where there is a relatively long 'plateau' phase, after which the error grows linearly in time, in accordance with the analysis of \cite{frutos1997accuracy}. This plateau phase seems to last longer, the smaller the timestep $\dt$ is. This behavior, in our opinion, explains the different orders that we see here. In any case, it can be said that for HB-I2DRK3-2s, the relaxed scheme behaves significantly better than the unrelaxed one.

There are no SSP results, as we observed severe stability issues in this case. This can of course be explained from the relatively low stability angles as observed in Fig.~\ref{fig:stability-angles-2}.

\begin{figure}[!htbp]
     \centering
\ifthenelse{\boolean{compilefromscratch}}
    {
      \tikzsetnextfilename{kdv_convg_T10}
      \input{tikz_source/kdv_convg_T10.tex}
      \tikzsetnextfilename{kdv_convg_T360}
      \input{tikz_source/kdv_convg_T360.tex}
    }
    {
      \includegraphics{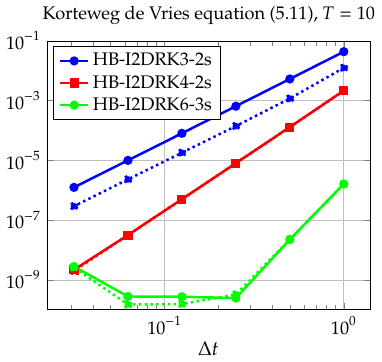}
      \includegraphics{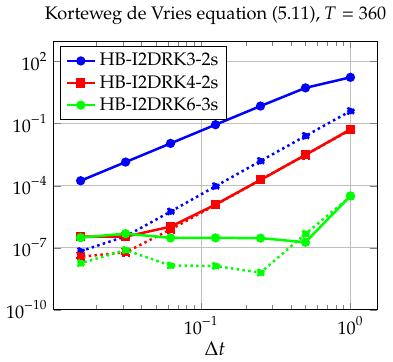}
    }

    \caption{Convergence results for the Korteweg de Vries equation \eqref{eq:kdv}. Shown are the outputs of some implicit collocation-based schemes. Left: Final time $T=10$, right: final time $T=360$.
    Dotted lines correspond to the relaxed schemes, straight lines to the baseline schemes. For the HB-I2DRK4-2s scheme, baseline and relaxation scheme lie on top of each other; for the HB-I2DRK6-3s scheme, they are fairly close to each other. }
    \label{fig:KdVConvg}
\end{figure}

\begin{figure}[!htbp]
    \centering
\ifthenelse{\boolean{compilefromscratch}}
    {
      \tikzsetnextfilename{kdv_error_baseline}
      \input{tikz_source/kdv_error_baseline.tex}
      \tikzsetnextfilename{kdv_error_relaxed}
      \input{tikz_source/kdv_error_relaxed.tex}
    }
    {
      \includegraphics{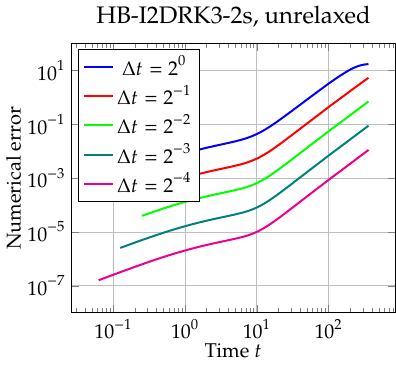}
      \includegraphics{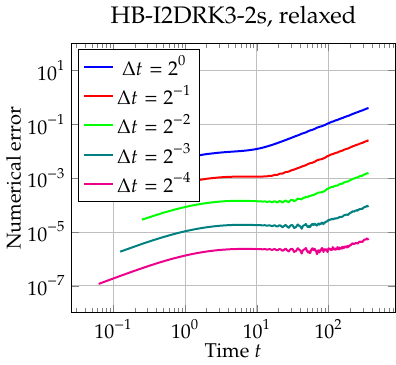}
    }

        \caption{Temporal evolution of the numerical error for the Korteweg de Vries equation \eqref{eq:kdv} discretized using the HB-I2DRK3-2s scheme for various timesteps $\dt$. Final time is $T=360$, which corresponds to three temporal periods of the solution. Left: Baseline scheme, right: relaxed scheme. The fact that the curves start at different positions on the $t-$axis is due to the fact that $t=0$ cannot be represented on this log-log-plot; the first data point is then at $\dt$.
        }\label{fig:kdv_errorgrowth}
    \end{figure}

%% file: sec_conout.tex
We have combined the recently developed relaxation approach to preserve
important functionals (``entropies'') of numerical solutions of ODEs obtained by high-order
multiderivative time integration methods with up to four temporal derivatives. Extending our previous work \cite{ranocha2023functional}, we have considered not only invariants but also
dissipated functionals. The latter required the development of appropriate
entropy estimates in Section~\ref{sec:entropy-estimates}.
In addition, we have analyzed stability properties of multiderivative
relaxation methods in Section~\ref{sec:stability}. The numerical results of
Section~\ref{sec:numerical_experiments} have demonstrated the robustness of
the approach as well as qualitative and quantitative improvements of
numerical solutions.

Future work in this direction will concentrate on energy-stable DG discretizations of the Cahn-Hilliard equation using a relaxation approach. Furthermore, numerical results have shown that some even-order collocation-based multiderivative schemes seem to handle quadratic entropies very well. A theoretical and practical study of the implications here will be performed.

%% file: funding.tex
HR was supported by the Deutsche Forschungsgemeinschaft
(DFG, German Research Foundation, project number 513301895)
and the Daimler und Benz Stiftung (Daimler and Benz foundation,
project number 32-10/22).